\documentclass[11pt]{amsart}
\usepackage{xcolor}
\usepackage{hyperref}
\hypersetup{
    colorlinks = false,
    linkbordercolor = {white},
    citebordercolor = {white}
    }
\usepackage[capitalise,noabbrev]{cleveref}
\usepackage[all]{xy}

\usepackage[utf8]{inputenc}
\usepackage{amsmath,amscd}
\usepackage{amssymb}
\usepackage{mathtools}
\usepackage{stmaryrd}
\usepackage{url}
\usepackage{tikz-cd}
\usepackage{enumitem}
\usepackage{fullpage}
\usepackage{musicography}

\usepackage{physics}

\usepackage{tikz}

\usepackage{enumitem,kantlipsum}

\usepackage{mathtools}

 \usepackage{comment}
 \usepackage{mathrsfs}
 \usepackage{tikz-cd}
\usepackage{relsize}

\author{Michael K. Brown}
\author{Prashanth Sridhar}

\newcommand{\Addresses}{{
	\vskip\baselineskip
  	\footnotesize
  	\noindent \textsc{Department of Mathematics and Statistics, Auburn University} \par\nopagebreak
	\noindent \textit{E-mail addresses:} \texttt{mkb0096@auburn.edu, pzs0094@auburn.edu}
 }}

\numberwithin{equation}{section}
\newtheorem{lemma}[equation]{Lemma}
\newtheorem{lem}[equation]{Lemma}

\newtheorem{prop}[equation]{Proposition}
\newtheorem{cor}[equation]{Corollary}
\newtheorem{conj}[equation]{Conjecture}

\newtheorem{claim*}{Claim}
\newtheorem{thm}[equation]{Theorem}

\theoremstyle{definition}
\newtheorem{defn}[equation]{Definition}
\newtheorem{dfn}[equation]{Definition}
\newtheorem{example}[equation]{Example}

\newtheorem{setup}[equation]{Setup}

\theoremstyle{remark}
\newtheorem{notation}[equation]{Notation}
\newtheorem{remark}[equation]{Remark}

\newtheorem{rem}[equation]{Remark}

\newcommand{\iso}{\xrightarrow{\simeq}}

\renewcommand{\k}{\Bbbk}

\renewcommand{\a}{\alpha}

\newcommand{\Hom}{\operatorname{Hom}}
\newcommand{\RHom}{\mathbf{R}\underline{\Hom}}

\newcommand{\D}{\msf{D}}
\newcommand{\kk}{\mathbf{k}}

\newcommand{\del}{\partial}
\newcommand{\Mod}{\operatorname{Mod}}
\def\mod{\operatorname{mod}}

\newcommand{\gr}{\operatorname{gr}}

\newcommand{\im}{\operatorname{Im}}

\def\lhom{\operatorname{\underline{Hom}}}

\newcommand{\dGamma}{\mathbf{R}\Gamma}

\newcommand{\cone}{\on{cone}}

\def\nc{\newcommand}
\nc{\on}{\operatorname}
\nc{\bideg}{\on{bideg}}
\nc{\xra}{\xrightarrow}
\def\phi{\varphi}
\nc\cB{\mathcal{B}}
\def\th{\on{th}}
\def\bu{\bullet}
\def\D{\on{D}}
\def\Db{\D^{\on{b}}}
\def\Df{\D^{\on{f}}_{\on{gr}}}
\nc{\into}{\hookrightarrow}
\nc{\onto}{\twoheadrightarrow}
\nc{\LL}{\mathbf{L}}
\nc{\RR}{\mathbf{R}}
\nc{\Perf}{\on{Perf}_{\on{gr}}}
\nc{\nat}{\natural}
\nc{\tors}{\on{tors}}
\nc{\Tors}{\on{Tors}}
\def\mod{\on{mod}}
\def\Mod{\on{Mod}}
\nc{\qgr}{\on{qgr}}
\nc{\Qgr}{\on{Qgr}}
\nc{\fQgr}{\on{Qgr}^{\on{f}}}
\nc{\colim}{\on{colim}}
\def\Z{\mathbb{Z}}
\nc{\Ext}{\on{Ext}}
\def\Dsing{\D_{\on{gr}}^{\on{sg}}}
\def\Dsg{\D_{\on{gr}}^{\on{sg}}}
\def\Dbgr{\Db_{\on{gr}}}
\nc{\om}{\omega}
\nc{\w}{\widetilde}
\nc{\PP}{\mathbb{P}}
\nc{\mf}{\on{mf}}
\nc{\OO}{\mathcal{O}}
\nc{\Proj}{\on{Proj}}
\nc{\Qcoh}{\on{Qcoh}}
\nc{\coh}{\on{coh}}
\nc{\Tor}{\on{Tor}}
\nc{\Modf}{\Mod^{\on{f}}}
\def\op{\on{op}}
\nc{\ce}{\coloneqq}
\def\co{\colon}
\def\k{\kk}
\nc{\Com}{\on{Com}}
\nc{\A}{\mathcal{A}}
\nc{\B}{\mathcal{B}}
\nc{\C}{\mathcal{C}}
\nc{\I}{\mathcal{I}}
\nc{\M}{\mathcal{M}}
\nc{\Sh}{\on{Sh}}
\nc{\QCoh}{\on{Qcoh}}
\nc{\Coh}{\on{coh}}
\nc{\fQCoh}{\QCoh^{\on{f}}}
\nc{\ov}{\overline}
\nc{\End}{\on{\underline{End}}}
\def\MR#1{}
\def\S{\mathcal{S}}
\def\P{\mathcal{P}}
\nc{\Qgrf}{\Qgr^{\on{f}}}
\nc{\uHom}{\underline{\Hom}}
\nc{\Inj}{\mathrm{Inj}}
\nc{\proj}{\mathrm{Proj}}
\nc{\spec}{\mathrm{Spec}}
\nc{\xla}{\xleftarrow}
\nc{\Dqgr}{\D_{\qgr}}
\nc{\DQgr}{\D_{\Qgr}}
\nc{\cK}{\mathcal{K}}
\nc{\from}{\leftarrow}
\nc{\cd}{\on{cd}}
\nc{\N}{\mathcal{N}}

\usepackage{comment}

\begin{document}
\title{Orlov's Theorem for dg-algebras}
\thanks{The first author was partially supported by NSF grant DMS-2302373.}

\begin{abstract}
A landmark theorem of Orlov relates the singularity category of a graded Gorenstein algebra to the derived category of the associated noncommutative projective scheme. We generalize this theorem to the setting of differential graded algebras. As an application, we obtain new cases of the Lattice Conjecture in noncommutative Hodge theory.  
\end{abstract}

\thanks{{\em Mathematics Subject Classification} 2020: 14F08.}

\numberwithin{equation}{section}

\maketitle

\setcounter{section}{0}

\section{Introduction}
Let $\k$ be a field, $f_1, \dots, f_c \in \k[x_0, \dots, x_n]$ a regular sequence of homogeneous forms, $X \subseteq \PP^n$ the associated projective complete intersection, and $R = \k[x_0, \dots, x_n] / (f_1, \dots, f_c)$ its homogeneous coordinate ring. A fundamental result of Orlov~\cite{Orlov2009} exhibits a close relationship between two categories of great interest in algebraic geometry and commutative algebra: the bounded derived category $\Db(X)$ of coherent $\OO_X$-modules and the singularity category $\Dsg(R)$, i.e. the quotient of the bounded derived category $\Dbgr(R)$ of $\Z$-graded $R$-modules by the subcategory $\Perf(R)$ of perfect complexes. An important special case of Orlov's Theorem is stated as follows:

\begin{thm}[\cite{Orlov2009}]
\label{thm:cy}
If $X$ is Calabi-Yau, then the categories $\Db(X)$ and $\Dsg(R)$ are equivalent. 
\end{thm}

Theorem~\ref{thm:cy} is a surprising and powerful result, with many applications across algebraic geometry, commutative algebra, and  mirror symmetry, e.g.~\cite{ADS, BD, BFKorlov, BFK, BS, CIR, CT, DHL, EG, Fan, FU, HLS, isik, segal, sheridan, shipman, toda}. It is a mathematical incarnation of a phenomenon in physics called the Landau-Ginzburg/Calabi-Yau correspondence~\cite{witten}. Theorem~\ref{thm:cy} has also been extended by Baranovsky-Pecharich to Calabi-Yau hypersurfaces in toric varieties~\cite{BP} and by Hirano to  gauged Landau-Ginzburg models~\cite{hirano}.

In this paper, we ask: does Theorem~\ref{thm:cy} extend to the case where $X = V(f_1, \dots, f_c)$ is any projective variety, rather than a complete intersection? We show that the answer is ``yes", provided that one replaces $\Dsg(R)$ with the singularity category of the Koszul complex on $f_1, \dots, f_c$, and one replaces $\Db(X)$ with the bounded derived category of the sheaf of dg-algebras on $\PP^n$ associated to the Koszul complex. In fact, Orlov realizes Theorem~\ref{thm:cy} as a special case of a far more general statement about Gorenstein algebras, and the main goal of this paper is to generalize this result to Gorenstein dg-algebras.

In order to discuss our results in detail, we begin by stating Orlov's Theorem in its full generality. To do so, we must introduce some notation and terminology.
Let $A = \bigoplus_{i \ge 0} A_i$ be a graded (not necessarily commutative) Gorenstein $\k$-algebra such that $A_0 = \k$. The complete intersection ring $R$ above is an example of such an algebra; we discuss the Gorenstein condition in detail in Section~\ref{sec:gorenstein}.
Since $A$ is Gorenstein, we have an isomorphism $\RHom_A(\k, A) \cong \k(a)$ in $\Dbgr(A)$ (up to a cohomological shift) for some $a \in \Z$; see Notation~\ref{conventions} for our conventions concerning grading twists and cohomological shifts. The integer $a$ is called the \emph{Gorenstein parameter of $A$}. For instance, the Gorenstein parameter of the complete intersection ring $R = \k[x_0, \dots, x_n] / (f_1, \dots, f_c)$ is $n+1 - \sum_{i = 1}^c \deg(f_i)$. Let $\D_{\qgr}(A)$ denote the quotient of $\Dbgr(A)$ by the subcategory of complexes whose cohomology is finite dimensional over~$\k$. That is, $\D_{\qgr}(A)$ is the derived category of the noncommutative projective scheme associated to $A$, in the sense of Artin-Zhang~\cite{AZ}. When $A$ is commutative and generated in degree 1, a classical theorem of Serre implies that $\D_{\qgr}(A)$ is equivalent to the bounded derived category of the projective scheme $\Proj(A)$.  

Orlov's Theorem, in its full form, is stated as follows:

\begin{thm}[\cite{Orlov2009} Theorem 2.5]
\label{thm:orlov}
Let $\kk$, $A$, and $a$ be as above, and let $q \co \Dbgr(A) \to \Dsing(A)$ and $\pi \co \Dbgr(A) \to \Dqgr(A)$ denote the canonical functors. The objects $\pi A(j) \in \Dqgr(A)$ and $q \k(j) \in \Dsg(A)$ are exceptional for all $j \in \Z$, and we have:
\begin{enumerate}
\item If $a > 0$, then for each $i \in \Z$, there is a fully faithful functor $\Phi_i \colon \Dsing(A) \to \Dqgr(A)$ and a semiorthogonal decomposition
$
\Dqgr(A) = \langle \pi A(-i - a + 1), \dots, \pi A(-i), \Phi_i\Dsing(A) \rangle.
$
\item If $a < 0$, then for each $i \in \Z$, there is a fully faithful functor $\Psi_i \colon \Dqgr(A) \to  \Dsing(A)$ and a semiorthogonal decomposition
$
\Dsing(A) = \langle q \kk(-i), \dots, q \kk (-i+a+1), \Psi_i \Dqgr(A)\rangle.
$
\item If $a = 0$, then there is an equivalence $\Dsing(A) \xra{\simeq} \Dqgr(A)$.
\end{enumerate}
\end{thm}

To realize Theorem~\ref{thm:cy} as a special case of Theorem~\ref{thm:orlov}, observe that when $A$ is the complete intersection ring $R$ above, we have $\Dqgr(A) = \Db(X)$; and if $X$ is Calabi-Yau, then, by definition, its canonical sheaf $\omega_X = \OO_X(\sum_{i = 1}^c d_i - n - 1)$ is trivial, i.e. $a = 0$.
This general form of the theorem has influenced not only algebraic geometry and commutative algebra, but representation theory as well, e.g. \cite{BIY,HIMO, IY, KLM, KMV, MU}. To illustrate the strength of Theorem~\ref{thm:orlov}, let us consider two simple families of examples. If $A$ has finite global dimension, then $\Dsing(A) = 0$; \Cref{thm:orlov} then implies that $a \ge 0$, and taking $i = 0$ gives the full exceptional collection $\Dqgr(A) = \langle\pi A(-a + 1), \dots, \pi A  \rangle$. In particular, if $A = \kk[x_0, \dots, x_n]$, and the degree of each variable is 1, then \Cref{thm:orlov} recovers the Beilinson exceptional collection $\Db(\PP^n) = \langle \OO(-n), \dots, \OO \rangle$~\cite{beilinson}. At the opposite extreme, if $\dim_\kk A < \infty$, then $\Dqgr(A)$ vanishes. \Cref{thm:orlov} implies in this case that $a \le 0$, and taking $i = 0$ gives the full exceptional collection $\Dsing(A) = \langle q\kk, \dots, q\kk(a + 1)\rangle $. 

\medskip

Our goal is to generalize Theorem~\ref{thm:orlov} to the setting of dg-algebras. Let $A$ be a differential bigraded $\kk$-algebra, i.e. a dg-algebra with both a cohomological grading and an ``internal" grading; see Definition~\ref{def:dga} for the precise definition and Notation~\ref{conventions} for an explanation of our indexing conventions. Denote by $\Dbgr(A)$ the bounded derived category of differential bigraded $A$-modules.
We define the quotients $\Dsg(A)$ and $\Dqgr(A)$ of $\Dbgr(A)$ exactly as above; when $A$ is as in Setup~\ref{setup2} below, the category $\Dqgr(A)$ is equivalent to $\Db(\A)$, where $\A$ is the sheaf of dg-algebras on $\Proj(A^0)$ associated to $A$ (Theorem~\ref{prop:sheaf}(2)).  When $A$ is Gorenstein (Definition~\ref{def:gorenstein}), there is an isomorphism $\RHom_A(\kk, A) \cong \kk(a)$ in $\Dbgr(A)$ (up to cohomological shift) for some $a \in \Z$, the Gorenstein parameter of $A$. The following generalization of Theorem~\ref{thm:orlov} is our main result:

\begin{thm}
\label{main}
Let $\kk$ be a field and $A$ a differential bigraded $\kk$-algebra (Definition~\ref{def:dga}) satisfying the conditions in Setup~\ref{setup}. Suppose $A$ is Gorenstein (Definition~\ref{def:gorenstein}) with Gorenstein parameter $a$, and assume the cohomological degree zero component $A^0$ of $A$ is either Gorenstein or strictly commutative (i.e. $xy = yx$ for all $x, y \in A^0$). Let $q \co \Dbgr(A) \to \Dsing(A)$ and $\pi \co \Dbgr(A) \to \Dqgr(A)$ denote the canonical functors. The objects $\pi A(j) \in \Dqgr(A)$ and $q \k(j) \in \Dsg(A)$ are exceptional for all $j \in \Z$, and we have:
\begin{enumerate}
\item If $a > 0$, then for each $i \in \Z$, there is a fully faithful functor $\Phi_i \colon \Dsing(A) \to \Dqgr(A)$ and a semiorthogonal decomposition
$
\Dqgr(A) = \langle \pi A(-i - a + 1), \dots, \pi A(-i), \Phi_i\Dsing(A) \rangle.
$
\item If $a < 0$, then for each $i \in \Z$, there is a fully faithful functor $\Psi_i \colon \Dqgr(A) \to  \Dsing(A)$ and a semiorthogonal decomposition
$
\Dsing(A) = \langle q \kk(-i), \dots, q \kk (-i+a+1), \Psi_i \Dqgr(A)\rangle.
$
\item If $a = 0$, then there is an equivalence $\Dsing(A) \xra{\simeq} \Dqgr(A)$.
\end{enumerate}

\end{thm}

In fact, Theorem~\ref{main} holds in greater generality; see Remark~\ref{rem:bidual}. Let us describe how Theorem~\ref{main} gives the desired extension of Theorem~\ref{thm:cy} to an arbitrary projective variety $V(f_1, \dots, f_c) \subseteq \PP^n$. The Koszul complex on  $f_1, \dots, f_c$ is a dg-algebra satisfying all of the conditions in Theorem~\ref{main}, with Gorenstein parameter $n+1 - \sum_{i = 1}^c \deg(f_i)$~
(Example~\ref{ex:koszulparam}). Theorem~\ref{main} therefore implies that, when the ``Calabi-Yau" condition $n+1 = \sum_{i = 1}^c \deg(f_i)$ holds, we have $\Dqgr(K) \simeq \Dsg(K)$. When $f_1, \dots, f_c$ is a regular sequence, the categories $\Dqgr(K)$ and $\Dsg(K)$ coincide with those appearing in Theorem~\ref{thm:cy}. In general,  $\Dqgr(K)$ is equivalent to the bounded derived category of the sheafified Koszul complex over $\PP^n$ (Theorem~\ref{prop:sheaf}(2)).

\medskip
If $\dim_\kk H(A) < \infty$, then $\Dqgr(A) = 0$. As a consequence of  Theorem~\ref{main} (and Remark~\ref{rem:bidual}), we therefore have:

\begin{cor}
\label{cor:exc}
Let $\k$ be a field and $A$ a dg-$\k$-algebra as in Setup~\ref{setup} that is Gorenstein (Definition~\ref{def:gorenstein}) with Gorenstein parameter $a$. Recall that $q \co \Dbgr(A) \to \Dsg(A)$ denotes the canonical functor. If $\dim_\k H(A) < \infty$, then $a \le 0$, and, for all $i \in \Z$, the singularity category $\Dsg(A)$ is generated by the exceptional collection $q \k(-i), \dots, q \k(-i + a + 1)$. 
\end{cor}

Corollary~\ref{cor:exc} yields a host of examples of singularity categories of dg-algebras with full exceptional collections; such singularity categories may also be identified with stable categories of Cohen-Macaulay modules by (a bigraded version of) a result of Jin~\cite[Theorem 0.3(4)]{Jin_2020}. For instance, if $f_1, \dots, f_c \in \k[x_0, \dots, x_n]$ generate an $(x_0, \dots, x_n)$-primary ideal, then, by Corollary~\ref{cor:exc}, the singularity category of the Koszul complex on this sequence is generated by an exceptional collection of $\sum_{i = 1}^c \deg(f_i) - n - 1$ objects.  
(By a result of Raedschelders-Stevenson, every dg-algebra as in Corollary~\ref{cor:exc} is quasi-isomorphic to a dg-algebra that is finite dimensional as a $\k$-vector space~\cite[Corollary 3.12]{RS}.) However, if $A$ is as in Theorem~\ref{main}, and $\Dsing(A) = 0$, then it follows from a result of J{\o}rgensen~\cite{Jorgensen_amplitude} that $A$ is concentrated in cohomological degree 0: see Remark~\ref{rem:jorgensen}. 

As in the case of Theorem~\ref{thm:orlov}, the exceptional collection appearing in Theorem~\ref{main}(2) need not be strong. However, while the exceptional collection in Theorem~\ref{thm:orlov}(1) is always strong, the exceptional collection in Theorem~\ref{main}(1) is \emph{not} necessarily strong: see Remark~\ref{rem:strong}, where we describe exactly when the exceptional collections in Theorem~\ref{main}(1) and (2) are strong. 

As an application, we show that the Lattice Conjecture (\Cref{conj:lattice}) in noncommutative Hodge theory holds for the bounded derived and singularity categories (of both $\Z^2$-graded and $\Z$-graded modules) associated to dg-algebras as in Corollary~\ref{cor:exc}: see Theorem~\ref{thm:lattice} for the precise statement, and see Section~\ref{sec:lattice} for background on this conjecture.

Let us give an overview of the paper. In Section~\ref{sec:dga}, we provide background on differential bigraded algebras and their derived categories. A key technical point in this section is the construction of semi-free resolutions with certain finiteness properties: see Proposition~\ref{prop:semi-free}. We also discuss in this section several families of examples of Gorenstein dg-algebras to which Theorem~\ref{main} applies: see Subsection~\ref{sec:gorenstein}.
Section~\ref{sec:qgr} is the technical heart of the paper, which is devoted to generalizing several aspects of Artin-Zhang's noncommutative projective geometry to the context of dg-algebras. We prove Theorem~\ref{main} in Section~\ref{sec:proof}, and we apply Corollary~\ref{cor:exc} to obtain new cases of the Lattice Conjecture in Section~\ref{sec:lattice}.

\subsection*{Acknowledgments.} 
We are grateful to Benjamin Briggs, Liran Shaul, Greg Stevenson, and Mark Walker for helpful conversations during the preparation of this paper. We also thank the referee for many suggestions that improved the paper, in particular for observing that Proposition~\ref{prop:faithful} can be generalized (see Remark~\ref{rem:bidual1}), and also for suggesting the question in Remark~\ref{rem:question}.

\begin{notation}
\label{conventions}
Throughout, $\kk$ denotes a field. We will consider bigraded $\kk$-vector spaces $V = \bigoplus_{(i, j) \in \Z^2} V_i^j$. The superscript will typically denote a cohomological grading, while the subscript will refer to an ``internal" grading. 
Given a homogeneous element $v \in V$, we let $\bideg(v) \in \Z^2$ denote its bidegree, while $\deg(v)$ (resp. $|v|$) denotes its internal (resp. cohomological) degree. Given an integer $m$, we will denote the $m^{\th}$ shift of $V$ in internal (resp. cohomological) degree by $V(m)$ (resp. $V[m]$). That is, $V(m)_i^j = V_{i + m}^j$, and $V[m]_i^j = V_i^{j + m}$.
\end{notation}

\section{Resolutions and derived categories over dg-algebras} 
\label{sec:dga}

In this section, we recall some background on differential bigraded algebras and their derived categories, and we establish several technical facts in this setting that we will need along the way.  

\begin{dfn}
\label{def:dga}
A \emph{differential bigraded $\kk$-algebra} is a bigraded $\kk$-algebra $A = \bigoplus_{(i, j) \in \Z^2} A_i^j$ equipped with a degree $(0, 1)$
$\kk$-linear map $\del_A$ that squares to 0 and satisfies the Leibniz rule:
$$
\del_A(xy) = \del_A(x)y + (-1)^{|x|}x\del_A(y).
$$
We say $A$ is \emph{graded commutative} if $xy = (-1)^{|x||y|}yx$ for all homogeneous $x, y \in A$. The \emph{opposite dg-algebra of $A$}, denoted $A^{\op}$, is the same as $A$ as a bigraded $\k$-module and has the same differential as $A$, and its multiplication is given by $x * y \ce (-1)^{|x||y|}yx$. A \emph{morphism} of differential bigraded $\kk$-algebras is a degree (0,0) $\k$-algebra homomorphism that commutes with differentials. Such a morphism is a \emph{quasi-isomorphism} if it induces an isomorphism on cohomology. 
\end{dfn}

Henceforth, we will refer to differential bigraded $\kk$-algebras as simply ``dg-algebras". We denote by $A^\nat$ the underlying bigraded $\kk$-algebra of a dg-algebra $A$. 

\begin{dfn}
Let $A$ be a dg-algebra. A right (resp. left) \emph{dg-$A$-module} is a bigraded right (resp. left) $A^\nat$-module $M = \bigoplus_{(i, j) \in \Z^2} M_i^j$ equipped with a degree $(0, 1)$ $\kk$-linear map $\del_M$ that squares to~0 and satisfies the Leibniz rule:
$$
\del_M(mx) = \del_M(m)x + (-1)^{|m|}m\del_A(x)
\quad
(\text{resp. }\del_M(xm) = \del_A(x)m + (-1)^{|x|}x\del_M(m)).
$$
All modules are assumed to be right modules unless otherwise noted. A \emph{submodule} of a dg-module $M$ is an $A^{\nat}$-submodule that is also a subcomplex. A dg-module $M$ is said to be \emph{finitely generated} if the underlying bigraded $A^{\nat}$-module is finitely generated. A \emph{morphism} of dg-$A$-modules is a degree $(0,0)$ $A^{\nat}$-linear map that commutes with differentials. Such a morphism is a \emph{quasi-isomorphism} if it induces an isomorphism on cohomology. A \emph{homotopy} between a pair of morphisms $f, g \colon M \to N$ of dg-$A$-modules is a degree $(0, -1)$ $A^{\nat}$-linear map $h$ such that $f - g = hd_M + d_Nh$. Let $\Hom_A(M, N)$ denote the set of dg-$A$-module morphisms from $M$ to $N$. Given a pair $A, B$ of dg-algebras, an \emph{$A$-$B$-bimodule} is a right $A^{\op} \otimes_{\kk} B$-module. 
\end{dfn}

\begin{notation}
\label{gradednotation}
Given a dg-$A$-module $M$, we set $M_i \coloneqq \bigoplus_{j \in \Z} M_i^j$ for all $i \in \Z$. Notice that $A_0$ is a dg-algebra (concentrated in internal degree 0), and each $M_i$ is a dg-$A_0$-module (concentrated in internal degree $i$). Similarly, we set $M^j \coloneqq \bigoplus_{i \in \Z} M^j_i$ for all $j \in \Z$.
\end{notation}

\begin{dfn}\label{def:connected}
 A dg-algebra $A$ is \emph{connected} if 
$A_0 = A_0^0 = \k$,
and $A_i^j = 0$ when $i<0$ or $j > 0$.
\end{dfn}

Notice that, if $A$ is connected, then $A_0 = \k$ is a dg-$A$-module.

\begin{rem}\label{rem: quasi-iso}
Suppose $A$ is a dg-algebra such that $A_0 = A_0^0 = \k$, and $H^j(A)_i = 0$ for $i < 0$ and $j > 0$. In this case, the dg-algebra $B \ce \bigoplus_{i \ge 0, j \le 0} A_i^j$ is connected, and the canonical map $B \into A$ is a quasi-isomorphism of dg-algebras. 
\end{rem}

\begin{remark}
\label{A0}
If $A$ is a connected dg-algebra, then $A^0$ is a nonnegatively $\Z$-graded ring, and the differential on any dg-$A$-module is $A^0$-linear. 
\end{remark}

\begin{remark}
\label{remark:smart}
Let $A$ be a connected dg-algebra and $M$ a dg-$A$-module. For $i \in \Z$, we have truncations
\begin{align*}
\sigma^{\ge i} M & \ce \left( \cdots \to 0 \to M^i / \im(d_M^{i - 1}) \to M^{i+1} \to \cdots \right), \\ 
\sigma^{\le i} M & \ce \left( \cdots \to M^{i-1} \to\ker(d^i_M) \to 0 \to \cdots \right).
\end{align*}
Since $A$ is connected, both $\sigma^{\ge i} M$ and $\sigma^{\le i} M$ are dg-$A$-modules. The natural map $M \to \sigma^{\ge i} M$ (resp. $\sigma^{\le i} M \to M$) induces an isomorphism on cohomology in degrees at least $i$ (resp. at most $i$). 
\end{remark}

Throughout the paper, we will work under the following setup:
\begin{setup}
\label{setup}
Let $A$ be a dg-algebra (in the sense of Definition~\ref{def:dga}) that is connected (Definition~\ref{def:connected}) and such that $H^0(A)$ is Noetherian and the total cohomology algebra $H(A)$ is finitely generated as an $H^0(A)$-module. 
\end{setup}

\begin{example}
\label{ex:koszulsetup}
Let $S$ be a commutative, nonnegatively $\Z$-graded $\kk$-algebra such that $S_0 = \k$. Let $f_1, \dots, f_c \in S$ be homogeneous of positive degree, and assume $S / (f_1, \dots, f_c)$ is Noetherian. The Koszul complex $K$ on $f_1, \dots, f_c \in S$ is a dg-algebra with all of the properties in Setup~\ref{setup}. In more detail: the underlying bigraded module of $K$ is $\bigwedge_SF$, where $F$ is a bigraded free $S$-module with basis $e_1,\dots,e_c$ such that $\bideg(e_i)=(\deg(f_i),-1)$. The algebra structure is given by the exterior product, and the differential is given by sending $e_{i_1}\wedge \dots \wedge e_{i_j}$ to $\sum_{l=1}^j (-1)^{l-1}f_{i_l}e_{i_1}\wedge \cdots \wedge \widehat{e_{i_l}} \wedge \cdots \wedge e_{i_j}$. See Section~\ref{sec:gorenstein} for several additional examples of dg-algebras as in Setup~\ref{setup}.
\end{example}

The following observation is elementary:

\begin{prop}
\label{prop:elementary}
If $A$ is as in Setup~\ref{setup}, then $\dim_\k H(A)_i < \infty$ for all $i$ (see Notation~\ref{gradednotation}). If the homogeneous maximal ideal of $A$ is finitely generated, then $\dim_\k A_i <~\infty$. 
\end{prop}

Given a dg-algebra $A$, a right dg-$A$-module $M$, and a left dg-$A$-module $N$, the tensor product $M \otimes_A N$ is a dg-$\kk$-module with differential $m \otimes n \mapsto d_M(m) \otimes n + (-1)^{|m|}m \otimes d_N(n)$. If $M$ (resp. $N$) is an $A$-$A$-bimodule, then $M \otimes_A N$ is a left (resp. right) dg-$A$-module.
Similarly, given right dg-$A$-modules $M$ and $N$, we may form the internal Hom object $\underline{\Hom}_A(M, N)$, which is a dg-$\kk$-module with underlying $\kk$-module $\lhom_{A^{\nat}}(M, N)$ and differential $\alpha \mapsto d_N \alpha - (-1)^{|\alpha|}\alpha d_M$. A map in $\underline{\Hom}_A(M, N)$ of bidegree $(0,0)$ is a cocycle if and only if it is a morphism of dg-$A$-modules, and it is a coboundary if and only if it is a null-homotopic such morphism. Notice the distinction between $\underline{\Hom}_A(M, N)$ and $\Hom_A(M, N)$; the latter is the set of bidegree $(0,0)$ cocycles in the former. If $M$ (resp. $N$) is an $A$-$A$-bimodule, then $\underline{\Hom}_A(M, N)$ is a right (resp. left) dg-$A$-module.

\subsection{Derived categories}
\label{sec:derived}
For a dg-algebra $A$, let $\Mod_{\on{gr}}(A)$ denote the category of dg-$A$-modules, $\Modf_{\on{gr}}(A)$ the full subcategory of $\Mod_{\on{gr}}(A)$ given by dg-$A$-modules $M$ such that $H(M)$ is finitely generated over $H(A)$, and $\mod_{\on{gr}}(A)$ the full subcategory of $\Modf(A)$ given by dg-$A$-modules that are finitely generated over $A$; the subscript $``\gr"$ indicates that we consider dg-modules that have both a cohomological and internal grading. We form the (triangulated) derived categories $\D_{\on{gr}}(A)$, $\Df(A)$, and $\D^b(A)$ by inverting quasi-isomorphisms in $\Mod_{\on{gr}}(A)$, $\Modf_{\on{gr}}(A)$, and $\mod_{\on{gr}}(A)$, respectively. For a construction of the derived category of a dg-algebra, see \cite[\href{https://stacks.math.columbia.edu/tag/09KV}{Tag 09KV}]{stacks-project}.

\begin{example}
\label{ex:trivial}
Let $A$ be a dg-algebra, and assume $|a| = 0$ for all $a \in A$; that is, assume the cohomological grading on $A$ is trivial. In particular, $A$ has trivial differential. Let $A'$ denote the $\Z$-graded algebra obtained from $A$ by forgetting its (trivial) cohomological grading. Notice that, for all $M \in \Mod(A)$ and $j \in \Z$, the $k$-vector space $M^j$ (see Notation~\ref{gradednotation}) is naturally a graded $A'$-module. Letting $\Com(A')$ denote the category of complexes of $\Z$-graded $A'$-modules, we therefore have an isomorphism of categories 
$
F \co \Mod_{\on{gr}}(A) \xra{\cong} \Com(A')
$
that sends a dg-module $(M, d_M)$ to the complex $F(M)$ with $F(M)^j = M^j$ and differential $d_M$. Under this isomorphism, $\Modf(A)$ corresponds to complexes whose total cohomology is finitely generated over $A'$, and $\mod(A)$ corresponds to bounded complexes of finitely generated graded $A'$-modules.
\end{example}

\begin{remark}
\label{rem:bounded}
If $A$ is as in Setup~\ref{setup}, then every object $M \in \Df(A)$ satisfies $H^i(M) = 0$ for $|i| \gg 0$.
\end{remark}

\begin{defn}
Let $A$ be a dg-algebra. A dg-$A$-module $L$ is \emph{$K$-projective} (resp. \emph{$K$-injective}) if the dg-$\kk$-module $\underline{\Hom}_A(L, N)$ (resp. $\underline{\Hom}_A(N, L))$ is exact for all exact right dg-$A$-modules $N$. We say $L$ is \emph{$K$-flat} if $L \otimes_A N$ is exact for all exact left dg-$A$-modules $N$. A \emph{$K$-projective resolution} of a dg-$A$-module $M$ is a quasi-isomorphism $P \xra{\simeq} M$, where $P$ is $K$-projective. $K$-injective and $K$-flat resolutions are defined similarly; such resolutions exist by \cite[Section 3]{Keller1994}.
\end{defn}

Given a pair $M$ and $N$ of dg-modules over a dg-algebra $A$, $K$-flat resolutions $F^M$ and $F^N$ of $M$ and $N$, a $K$-projective resolution $P$ of $M$, and a $K$-injective resolution $I$ of $N$, the derived tensor product of $M$ and $N$ and derived Hom from $M$ to $N$ may be modeled as follows:
$$
M \otimes_A^{\LL} N \cong F^M \otimes_A N \cong M \otimes_A F^N, \quad  \RHom_A(M, N) \cong \underline{\Hom}_A(P, N) \cong \underline{\Hom}_A(M, I);
$$
where the isomorphisms are in the derived category $\D_{\on{gr}}(A)$. We have $$
\underline{\Ext}^i_A(M, N) \coloneqq H^i\RHom_A(M, N),
 \quad \text{and} \quad 
 \Tor^A_i(M, N) \coloneqq H^{-i}(M \otimes_A^{\LL} N).
$$

Tor and $\underline{\Ext}$ over a dg-algebra can be computed via semi-free resolutions: 

\begin{dfn}
\label{def:semi-free}
A dg-$A$-module $G$ is called \emph{free} if it is isomorphic, as a dg-module, to a direct sum of copies of $A(i)[j]$ for various $(i, j) \in \Z^2$.
The dg-$A$-module $G$ is called \emph{semi-free} if it may be equipped with an increasing, exhaustive filtration $F^\bu G$ by dg-submodules such that $F^i G = 0$ for $i < 0$ and each dg-module $F^i G / F^{i-1} G$ is free. Given a dg-$A$-module $M$, a \emph{semi-free resolution} of $M$ is a quasi-isomorphism $G \xra{\simeq} M$, where $G$ is semi-free. A dg-$A$-module $M$ is called \emph{perfect} if it admits a semi-free resolution $G$ that is finitely generated as an $A^\nat$-module. Let $\Perf(A)$ be the subcategory of $\Dbgr(A)$ given by perfect objects.
\end{dfn}

\begin{remark}
\label{rem:perf}
It is elementary to show that $\Perf(A)$ is the thick subcategory of $\Dbgr(A)$ generated by the objects $\{A(j)\}_{j \in \Z}$.
\end{remark} 

Given a dg-algebra $A$, it is well-known that semi-free dg-$A$-modules are $K$-projective, and $K$-projective dg-$A$-modules are $K$-flat. In particular, if $M$ and $N$ are dg-$A$-modules, and $F \iso M$ is a semi-free resolution, then 
$$
\underline{\Ext}^i_A(M, N) \cong H^i(\underline{\Hom}_A(F, N)), \quad \text{and} \quad \Tor^A_i(M, N) \cong H^{-i}(F \otimes_A N).
$$
We will need the existence of semi-free resolutions with certain finiteness properties:

\begin{prop}[cf. \cite{avramov1997differential} Chapter 5, Theorem 2.2]
\label{prop:semi-free}
Let $A$ be as in Setup~\ref{setup}. 
If $M \in \Modf_{\gr}(A)$, then there exists a semi-free resolution $G$ of $M$ with the following properties:
\begin{enumerate}
\item $G_i = 0$ for $i \ll 0$.
\item $G^j = 0$ for $j \gg 0$.
\item If the homogeneous maximal ideal of $A$ is finitely generated, then $\dim_\k G_i < \infty$ for all $i \in \Z$.
\item Given $j \in \Z$, let $G^{\le j}$ denote the $A^\nat$-submodule $\bigoplus_{\ell \le j} G^\ell$ of $G$; the $A^\nat$-module $G / G^{\le j}$ is finitely generated.
\end{enumerate}
(See Notation~\ref{gradednotation} for the definitions of $G_i$ and $G^j$.)
\end{prop}

\begin{proof}
We build $G$ inductively. This is trivial if $M$ is exact, so assume otherwise. Since $A$ is connected, $H(M)$ is finitely generated over $H(A)$, and $H(A)$ is Noetherian; we have $H(M)_i = 0$ for $i \ll 0$, $H(M)^j = 0$ for $j \gg 0$, and $\dim_\k H(M)_i < \infty$ for all $i \in \Z$. Let $m \coloneqq \min\{i \text{ : } H(M)_i\neq 0\}$ and $n \ce \max\{j \text{ : } H(M)^j \ne 0\}$; replacing $M$ with $M(m)[n]$, we may assume that $m = n = 0$. Consider $H(M)_0$ as a bigraded $\k$-vector space (concentrated in internal degree $0$), and let $F^0 G$ be the free dg-$A$-module $H(M)_0 \otimes_{\k} A$. Choose a homogeneous basis $x_1, \dots, x_t$ of $H(M)_0$ and lifts $\widetilde{x_i}$ to cocycles in $M$.
Notice that $F^0G$ has a basis $x_1 \otimes 1, \dots, x_t \otimes 1$ as a free $A^\nat$-module, and $\bideg(x_i \otimes 1) = (0, j_i)$ for some $j_i \le 0$. Let $\varepsilon_0 \colon F^0 G \to M$ be the morphism of dg-$A$-modules given by $x_i \otimes 1 \mapsto \widetilde{x_i}$. The map $\varepsilon_0$ induces an isomorphism on cohomology in internal degrees $\le 0$, i.e. $\cone(\varepsilon_0)$ has cohomology concentrated in internal degrees $> 0$.  Now apply this same construction with $\cone(\varepsilon_0)$ playing the role of $M$ to build a free dg-$A$-module $F^{1} G$ and a map $\varepsilon_{1} \colon F^{1} G \to \cone(\epsilon_0)$ that is a quasi-isomorphism in internal degrees at most $1$. Iterating this procedure gives a semi-free resolution $G \iso M$ with the desired properties. Indeed, (1) and (2) follow since the summands $A(i, j)$ of $G$ all satisfy $i \le 0$ and $j \ge 0$. Observe that each $F^nG$ is finitely generated over $A$; and if $A(i, j)$ is a summand of $F^{n+1} G / F^n G$, and $A(i', j')$ is a summand of $F^nG$, then $i < i'$ and $j > j'$. That is, both components of the bidegrees of the generators of the free summands we add in each step of the construction strictly increase in absolute value. Part (3) therefore follows from Proposition~\ref{prop:elementary}, and (4) follows immediately as well. 
\end{proof}

As a consequence of Proposition~\ref{prop:semi-free}, we have:

\begin{prop}
\label{dbdf}
Let $A$ be as in Setup~\ref{setup}. The canonical map $\Dbgr(A) \to \Df(A)$ is an equivalence.
\end{prop}

\begin{proof}
Let $M \in \Df(A)$. Since $H(M)$ is finitely generated over $H(A)$, we may choose $J \ll 0$ so that $H^j(M) = 0$ for $j \le J$. Choose a semi-free resolution $G$ of $M$ as in Proposition~\ref{prop:semi-free}. 
Denote by $d$ the differential on $G$, and set $G' = G^{\le J} + d(G^{\le J})$, where we adopt the notation of Proposition~\ref{prop:semi-free}(4). Notice that $G'$ is a dg-submodule of $G$, and it is exact. The map $G \to G / G'$ is therefore a quasi-isomorphism, and $G/G'$ is finitely generated over $A^{\nat}$ by Property (4) from Proposition~\ref{prop:semi-free}. Thus, the canonical map $\Dbgr(A) \to \Df(A)$ is essentially surjective, and it is clear that it is fully faithful.
\end{proof}

\subsection{Gorenstein dg-algebras}
\label{sec:gorenstein}

\begin{defn}
\label{def:gorenstein}
Let $A$ be as in Setup~\ref{setup}. We say $A$ is \emph{Gorenstein} if:
\begin{enumerate}
\item The functor $\RHom_A( - , A)$ maps $\Dbgr(A)$ to $\Dbgr(A^{\op})$, and the functor $\RHom_{A^{\op}}( - , A)$ maps $\Dbgr(A^{\op})$ to $\Dbgr(A)$.
\item Given $M \in \Dbgr(A)$ and $N \in \Dbgr(A^{\op})$, the canonical maps
$$
M \to \RHom_{A^{\op}}(\RHom_A(M, A), A) \quad \text{and} \quad N \to \RHom_{A}(\RHom_{A^{\op}}(N, A), A)
$$
are isomorphisms in $\D_{\on{gr}}(A)$ and $\D_{\on{gr}}(A^{\op})$, respectively.
\item There is an isomorphism $\RHom_A(\k, A) \cong \k(a)[n]$ in $\D_{\on{gr}}(A)$ for some $a, n \in \Z$. 
\end{enumerate}
The integer $a$ in (3) is called the \emph{Gorenstein parameter of $A$}.
\end{defn}

\begin{remark}
\label{rem:gorenstein}
There are at least two other definitions of a Gorenstein dg-algebra in the literature: let us compare Definition~\ref{def:gorenstein} with these. 
\begin{enumerate}
\item By Remark~\ref{rem:bounded} and Proposition~\ref{dbdf}, the category $\Df(A)$ coincides with the category $\mathbf{fin}(A)$ defined by Frankild-J\o rgensen in \cite[Definition 1.9]{Frankild2003}. It follows that the combination of Conditions (1) and (2) in Definition~\ref{def:gorenstein} coincides with (a bigraded version of) the definition of a Gorenstein dg-algebra in \cite[Definition 2.1]{Frankild2003}. On the other hand, Condition~(3) in Definition~\ref{def:gorenstein} is (a bigraded version of) Avramov-Foxby's definition of a Gorenstein dg-algebra in \cite[Section 3]{AF}. We will need the features of both of these definitions throughout the paper, and so we opt for our definition of Gorenstein to be given by combining Frankild-J\o rgensen's and Avramov-Foxby's definitions. 
\item When $A$ is graded commutative, a (bigraded version of a) theorem of Frankild-Iyengar-J\o rgensen \cite[Theorem I]{FIJ} implies that the combination of Conditions (1) and (2) is equivalent to Condition (3), and so Definition~\ref{def:gorenstein} is equivalent to both Frankild-J\o rgensen's and Avramov-Foxby's definitions in this case. We do not know if the combination of (1) and (2) implies (3) in general. 
\end{enumerate}
\end{remark}

\begin{remark}
\label{rem:orlovgorenstein}
In \cite{Orlov2009}, a Noetherian graded $\k$-algebra $A = \bigoplus_{i \ge 0} A_i$ such that $A_0 = \k$ is said to be Gorenstein if $\RHom_A(\k, A) \simeq \k(a)[-n]$ for some $a\in \Z$, and $A$ has finite injective dimension $n$ as a right $A$-module. These conditions imply (1) - (3) in Definition~\ref{def:gorenstein} in this case.
\end{remark}

Propositions~\ref{prop:selfdual} and~\ref{cor:selfdual} below give conditions under which one can easily compute the Gorenstein parameter of a Gorenstein dg-algebra. 

\begin{prop}
\label{prop:selfdual}
Let $A$ be as in Setup~\ref{setup}. Assume there are isomorphisms 
$$
\RHom_{A^0}(\k, A^0) \cong \k(a)[n] \quad \text{and} \quad \RHom_{A^0}(A, A^0) \cong A(s)[t]
$$
in $\D_{\on{gr}}(A)$ for some $a, n, s, t \in \Z$. We have $\RHom_A(\k, A) \cong \k(a - s)[n - t]$.
\end{prop}

\begin{proof}
There is an isomorphism in  $\D_{\on{gr}}(A)$ between $\RHom_A(\k, A)$ and
$$
 \RHom_A(\k, \RHom_{A^0}(A, A^0)(-s)[-t]) \cong \RHom_{A^0}(\k, A^0(-s)[-t]) \cong \k(a - s)[n - t].
$$
\end{proof}

\begin{prop}
\label{cor:selfdual}
Let $A$ be as in Setup~\ref{setup}. Assume $A$ is Gorenstein and graded commutative, and assume also that $A^0$ is Gorenstein.  
There is an isomorphism $\RHom_{A^0}(A, A^0) \cong A(s)[t]$ in $\D_{\on{gr}}(A)$ for some $s, t \in \Z$. In particular, the conclusion of Proposition~\ref{prop:selfdual} holds in this case. 
\end{prop}

\begin{proof}
This follows from (bigraded versions of) \cite[2.6]{FIJ} and \cite[Corollary 7.16]{yekutieli2013duality}.
\end{proof}

\begin{example}
\label{ex:koszulparam}
Let $S$ be a nonnegatively $\Z$-graded, strictly commutative $\k$-algebra such that $S_0 = \k$. Assume $S$ is Gorenstein. Let $K$ be the Koszul complex on a (not necessarily regular) sequence $f_1, \dots, f_c$ of homogeneous elements in $S$. The dg-algebra $K$ satisfies the conditions in Setup~\ref{setup} (see Example~\ref{ex:koszulsetup}), and a bigraded version of \cite[Theorem 4.9]{Frankild2003} implies that $K$ satisfies (1) and (2) in Definition~\ref{def:gorenstein}. Since $K$ is graded commutative, Remark~\ref{rem:gorenstein}(2) implies that $K$ is Gorenstein. Let $n$ be the injective dimension of $S$ over itself, $a$ the Gorenstein parameter of $S$, and $d = \sum_{i = 1}^c \deg(f_i)$. Noting that $\RHom_S(K, S) \cong K(d)[-c]$, \Cref{cor:selfdual} implies $\RHom_K(\k, K) \cong \k(a - d)[c - n]$. In particular, the Gorenstein parameter of $K$ is $a - d$. 

There is a more general construction of Koszul complexes, due to Shaul \cite{Shaul_Koszul}, that yields a wider family of examples of Gorenstein dg-algebras. Let $A$ be as in Setup~\ref{setup}. Assume $A$ is graded commutative and that $a^2=0$ for all $a \in A^i$ with $i$ odd (this last condition is automatic when $\on{char}(\k) \ne 2$). Let $\ov{f_1},\dots,\ov{f_c}\in H^0(A)$ be homogeneous elements of positive internal degree, and choose homogeneous lifts $f_1, \dots, f_c \in A^0$. Equip $R \ce \kk[x_1, \dots, x_c]$ with the bigrading given by $\bideg(x_i) = (\deg(f_i), 0)$ and a trivial differential, so that $R$ is a dg-algebra, and $A$ is a right dg-$R$-module via the action $a \cdot x_i \ce af_i$. Let $T$ denote the Koszul complex on $x_1, \dots, x_c$ over $R$. The \textit{Koszul complex} $K(A;\ov{f_1},\dots, \ov{f_c})$ is defined to be the dg-algebra $A\otimes_{R} T$. By a bigraded version of \cite[Theorem 4.11(a)]{Shaul_Koszul} and \Cref{rem:gorenstein}(2), $K(A;\ov{f_1},\dots,\ov{f_c})$ is Gorenstein if $A$ is Gorenstein. 
\end{example}

\begin{example}
Let $\a: R\rightarrow S$ be a morphism of nonnegatively $\mathbb{Z}$-graded, strictly commutative, Noetherian $\kk$-algebras such that $R_0=\kk$ and $S_0$ is a field extension of $\kk$. Let $\mu^i_R$ and $\mu^i_S$ denote the $i^{\th}$ Bass numbers of $R$ and $S$. The map $\a$ is called \emph{Gorenstein} if $S$ has finite flat dimension over $R$, and there exists $d \ge 0$ such that $\mu_R^i = \mu_S^{i+d}$ for all $i \in \Z$~\cite[Section 4]{Fossum_Foxby}. Let $T$ denote the minimal free $R$-resolution of $\k$, which is a dg-algebra by work of Tate and Gulliksen~\cite{tate, gulliksen}. The \textit{dg-fiber of $\a$} may be modeled by the graded commutative dg-algebra $A \ce T \otimes_R S$. By (a bigraded version of) \cite[Theorem 4.4]{AF} and \Cref{rem:gorenstein}(2), if $\a$ is  Gorenstein, then $A$ is a Gorenstein dg-algebra. For instance, if $R$ is a complete intersection, and $S$ is a Gorenstein quotient of $R$ such that $\on{pd}_R S < \infty$, then the canonical map $R \to S$ is Gorenstein~\cite[Proposition 4.3, Corollary 7.3]{AF}. In this case, $T$ is the Shamash resolution of $\k$ over $R$, and we conclude that $T \otimes_R S$ is a Gorenstein dg-algebra. 
\end{example}

\begin{example}\label{Ex:gorenstein_tensor_product}
Let $R$ be a nonnegatively $\mathbb{Z}$-graded, strictly commutative, Gorenstein $\kk$-algebra of finite Krull dimension such that $R_0=\kk$. Suppose $S$ and $T$ are graded commutative Gorenstein dg-algebras equipped with morphisms $R \to S$, $R \to T$ of dg-algebras such that (a) $H^0(S)$ and $H^0(T)$ are essentially of finite type over $R$, and (b) $S$ and $T$ have finite flat dimension over $R$. Suppose $A$ is a connected, graded commutative dg-algebra that is quasi-isomorphic, as a dg-algebra, to $S \otimes^{\LL}_R T$; when $S$ and $T$ are concentrated in cohomological degree 0, such a dg-algebra $A$ always exists \cite[Proposition 6.1.4]{Avramov_infinite}. By (a bigraded version of) \cite[Theorem 4.4]{SHAUL_twisted_inverse_image} and \Cref{rem:gorenstein}(2), $A$ is a Gorenstein dg-algebra (this statement is a derived version of \cite[Theorem 2(1)]{Watanabe_Gorenstein}). As a concrete example, take $R$ to be the standard graded polynomial ring $\k[x_0, \dots, x_3]$, $I$ the ideal $(x_0x_3-x_1x_2,x_2^2-x_1x_3,x_1^2-x_0x_2, x_0^2x_1-x_2^2x_3, x_0^3-x_2^3) \subseteq R$, and $S = T = R/I$. Geometrically, $R/I$ is the intersection inside the quadric $V(x_0x_3 - x_1x_2) \subseteq \PP^3$ of the twisted cubic $V(x_0x_3-x_1x_2,x_1^2-x_0x_2,x_2^2-x_1x_3)$ and the subvariety $V(x_0x_3-x_1x_2, x_0^2x_1-x_2^2x_3, x_0^3-x_2^3)$. It follows from the Buchsbaum-Eisenbud structure theorem for codimension 3 Gorenstein ideals that $R/I$ is Gorenstein; indeed, $I$ is the ideal of $4^{\th}$ order Pfaffians of the alternating matrix 

\[\begin{pmatrix} 0 & x_2^2 & x_0^2 & x_1 & x_2 \\
-x_2^2 & 0 & 0 & x_0 & x_1 \\
-x_0^2 & 0 & 0 & x_2 & x_3 \\
-x_1 & -x_0 & -x_2 & 0 & 0 \\
-x_2 & -x_1 & -x_3 & 0 & 0
\end{pmatrix}.\] The minimal $R$-free resolution $F$ of $R/I$ is a connected, graded commutative dg-algebra \cite[Proposition 1.3]{BE_77}, and so the dg-algebra $A = F \otimes_R F$ is Gorenstein. We observe that $A$ is not quasi-isomorphic to an ordinary graded algebra, since $H^{-1}(A) = \Tor^R_1(R/I,R/I)= I/I^2\neq 0$.
\end{example}

\begin{example}\label{ex:dualizing_module}
    Let $R$ be a nonnegatively $\mathbb{Z}$-graded strictly commutative $\kk$-algebra with $R_0=\kk$. Suppose $R$ admits a dualizing complex $D$. Let $A:=R\ltimes D$ denote the \emph{trivial extension dg-algebra}, as defined in \cite[Definition 1.2]{Jorgensen_03}. Replacing $D$ with suitable twists and shifts, it follows from a bigraded version of \cite[Theorem 2.2]{Jorgensen_03} and \Cref{rem:gorenstein}(2),  that the dg-algebra $A$ is Gorenstein.
\end{example}

\begin{example}
\label{ex:nc}
    There is a version of the trivial extension dg-algebra in Example~\ref{ex:dualizing_module} that is built from certain (possibly noncommutative) dg-algebras. In detail: let $B$ be a dg-algebra as in Setup~\ref{setup}, and assume $\dim_\k H(B) < \infty$. For $b,d\gg 0$, the \textit{trivial extension dg-algebra} $A\ce B\oplus \Hom_{\kk}(B,\kk)(-b)[d]$ is a Gorenstein dg-algebra satisfying $\dim_{\kk}H(A)<\infty$. To see this, note that it is observed in \cite[Section 6]{Jin_2020} that there is an isomorphism $\Hom_{\kk}(A,\kk)\cong A(-a)[-n]$ in $\D(A)$ for some $a,n\in \mathbb{Z}$. By \cite[Proposition 2.6]{FIJ}, we conclude that $A$ satisfies conditions (1) and (2) of \Cref{def:gorenstein}. By adjunction, one also sees that $A$ satisfies condition (3) in \Cref{def:gorenstein}, and hence $A$ is Gorenstein. One may start with the case where $B$ is concentrated in cohomological degree 0 to inductively build a large family of Gorenstein dg-algebras using this construction. See also \cite[Theorem 4.5]{shaul2022finitisticdimensionconjecturedgrings} for a related result.
\end{example}

\section{Noncommutative algebraic geometry over a dg-algebra}
\label{sec:qgr}

One may associate to a (strictly) commutative, nonnegatively $\Z$-graded ring $A$ a projective scheme $\Proj(A)$. For instance, when $A = \kk[x_0, \dots, x_n]$, we have $\Proj(A) = \PP^n$. When $A$ is noncommutative, the Proj construction no longer makes sense. However, it follows from work of Artin-Zhang \cite{AZ} that the category of coherent sheaves on $\Proj(A)$ \emph{does} generalize to the noncommutative setting, allowing one to extend many homological aspects of projective geometry to the noncommutative world. In order to state and prove our main result, Theorem~\ref{main}, we must extend many of Artin-Zhang's results to dg-algebras: this is the goal of the present section.

 Let $A$ be as in Setup~\ref{setup}. Given $n \ge 1$, let $A^0_{\ge n}$ be the ideal $\bigoplus_{i \ge n} A^0_i \subseteq A^0$. 
 Denote by $\D^{\Tors}_{\on{gr}}(A)$ the thick subcategory of $\D_{\gr}(A)$ given by objects $M$ such that, for any class $m \in H(M)$, there exists an integer $n$ such that $m \cdot A^0_{\geq n}=0$. Let $\D^{\tors}_{\on{gr}}(A) \ce \D^{\Tors}_{\on{gr}}(A) \cap \Db(A)$. We define $\DQgr(A) \ce \D_{\on{gr}}(A) / \D^{\Tors}_{\on{gr}}(A)$, and we let $\on{D}_{\qgr}(A)$ denote the essential image of the fully faithful embedding $\Dbgr(A) / \D^{\tors}_{\on{gr}}(A)\hookrightarrow \DQgr(A)$. 

\begin{lemma}
\label{thick}
The category $\D^{\tors}_{\on{gr}}(A)$ is the thick subcategory of $\Dbgr(A)$ generated by $\k(i)$ for $i \in \Z$. 
\end{lemma}

\begin{proof}
Let $\mathcal{S}$ denote the thick subcategory of $\Dbgr(A)$ generated by $\k(i)$ for $i \in \Z$. We clearly have $\mathcal{S} \subseteq \D^{\tors}_{\on{gr}}(A)$. Let $M \in \D^{\tors}_{\on{gr}}(A)$. If $M$ is exact, then it is in $\D^{\tors}_{\on{gr}}(A)$, so assume otherwise. Suppose $M$ is concentrated in cohomological degree 0. In particular, $M$ is a finitely generated $A^0$-module and therefore a finite dimensional $\k$-vector space, say of dimension $d$. If $d = 1$, then $M \cong \kk(i)$ for some $i$ and is thus an object in $\mathcal{S}$. If $d > 1$,  then either $M$ is a sum of twists of $\k$, in which case $M \in \mathcal{S}$, or the exact sequence
$
0 \to M \cdot A^0_{\ge 1} \into M \onto M / M \cdot A^0_{\ge 1} \to 0
$
implies that $M \in \mathcal{S}$ by induction on~$d$. In general, let $r$ denote the \emph{amplitude} of $M$, i.e.
$
r = \max\{i - j \text{ : } H^i(M) \ne 0 \text{ and } H^j(M) \ne 0\};
$
we argue by induction on $r$. Suppose $r = 0$; without loss, assume $H^0(M) \ne 0$. By Remark~\ref{remark:smart}, we have an isomorphism $M \cong \sigma^{\ge 0}(\sigma^{\le 0} M) = H^0(M)$ in $\D(A)$, and so $M \in \mathcal{S}$ by the above arguments. Suppose $r > 0$, and let $m = \min\{ i \text{ : } H^i(M) \ne 0\}$. By Remark~\ref{remark:smart} and induction, the exact sequence
$
0 \to \sigma^{\le m} M \into M \to M / \sigma^{\le m} M \to 0
$
implies that $M \in \mathcal{S}$.
\end{proof}

 Given $M \in \D(A)$, we let $\widetilde{M}$ denote the corresponding object in $\DQgr(A)$. 
We have canonical triangulated functors
$
\Pi \co \D_{\on{gr}}(A) \to \DQgr(A)$ and  $\pi \co \Dbgr(A) \to \Dqgr(A)
$
given by $M \mapsto \widetilde{M}$. When $A = A^0$, it follows from \cite[Lemma 4.4.1]{krause} that the categories $\DQgr(A)$ and $\Dqgr(A)$ coincide with the derived categories of the abelian categories $\Qgr(A)$ and $\qgr(A)$ defined in \cite{AZ}. In particular, if $A$ is concentrated in cohomological degree 0, and $A$ is (strictly) commutative and generated in internal degree 1, then $\DQgr(A)$ (resp. $\Dqgr(A)$) is equivalent to the derived category of quasi-coherent (resp. coherent) sheaves on $\Proj(A)$.

\begin{remark}
Versions of the derived ``qgr"-construction for dga's have appeared before, for instance in Greenlees-Stevenson's definition of the cosingularity category \cite[Definition 9.7, Remark 9.9]{GS} and in Lu-Palmieri-Wu-Zhang's $A_{\infty}$-version of the qgr   construction \cite[Section 10]{Koszul_ainfinity}.
\end{remark}

\subsection{Derived global sections}

Let $A$ be as in Setup~\ref{setup}. For $n \in \Z$, define  $\D_{\on{gr}}(A)_{\ge n} \subseteq \D_{\on{gr}}(A)$ to be the subcategory given by objects $M$ such that $M_i = 0$ for $i < n$. We define $\D^{\Tors}_{\on{gr}}(A)_{\ge n} \subseteq \D^{\Tors}_{\on{gr}}(A)$, $\Dbgr(A)_{\ge n} \subseteq \Dbgr(A)$, and $\D^{\tors}_{\on{gr}}(A)_{\ge n} \subseteq \D^{\tors}_{\on{gr}}(A)$ similarly.
The canonical functors
$$
\D_{\on{gr}}(A)_{\ge n} / \D^{\Tors}_{\on{gr}}(A)_{\ge n} \to \DQgr(A), \quad \Dbgr(A)_{\ge n} / \D^{\tors}_{\on{gr}}(A)_{\ge n} \to \DQgr(A)
$$
are triangulated equivalences for all $n \in \Z$. The truncation functor $\tau_{\ge n} \co \D_{\on{gr}}(A) \to \D_{\on{gr}}(A)_{\ge n}$ is the right adjoint to the inclusion $\D_{\on{gr}}(A)_{\ge n} \into \D_{\on{gr}}(A)$, and its restriction $\Dbgr(A) \to \Dbgr(A)_{\ge n}$, which we also denote by $\tau_{\ge n}$, is the right adjoint of the inclusion $\Dbgr(A)_{\ge n} \into \Dbgr(A)$.

We now define a ``total derived global sections functor" $\dGamma_* \co \DQgr(A) \to \D_{\on{gr}}(A)$.

\begin{prop}
\label{prop:rightadjoint}
Let $A$ be as in Setup~\ref{setup}. The functor $\Pi \co \D_{\on{gr}}(A) \to \DQgr(A)$ admits a fully faithful right adjoint $\RR\Gamma_*$. Moreover, for each $n \in \Z$, the canonical functor $\Pi_n \co \D_{\on{gr}}(A)_{\ge n} \to \DQgr(A)$ admits a fully faithful right adjoint given by $\RR\Gamma_{\ge n} \ce \tau_{\ge n} \RR\Gamma_*$. 
\end{prop}

\begin{proof}
To show that $\Pi$ admits a right adjoint, we wish to apply \cite[Example 8.4.5]{neeman}. To do so, we will check that the categories $\D_{\on{gr}}(A)$ and $\DQgr(A)$ are locally small (meaning that morphisms between any pair of objects in these categories form sets), $\D_{\on{gr}}(A)$ is $\aleph_1$-perfectly generated (in the sense of \cite[Definition 8.1.2]{neeman}) and closed under arbitrary coproducts, and $\D^{\Tors}_{\gr}(A)$ is a localizing subcategory of $\D_{\on{gr}}(A)$ (the assumptions that $\D_{\on{gr}}(A)$ is $\aleph_1$-perfectly generated, admits arbitrary coproducts, and is locally small imply that $\D_{\on{gr}}(A)$ satisfy the representability theorem, in the sense of \cite[Definition 1.20]{neeman}: see \cite[Theorem 8.3.3]{neeman}). Local smallness of $\D_{\on{gr}}(A)$ is well-known and follows, for instance, from (a bigraded version of) \cite[\href{https://stacks.math.columbia.edu/tag/09KY}{Tag 09KY}]{stacks-project}. The set $S = \{A(i)[j] \text{ : } i, j \in \Z\}$ is $\aleph_1$-perfect in the sense of \cite[Definition 3.3.1]{neeman}, and it generates $\D_{\on{gr}}(A)$ in the sense of \cite[Definition 8.1.1]{neeman}; it follows that $\D_{\on{gr}}(A)$ is $\aleph_1$-perfectly generated. It is clear that $\D_{\on{gr}}(A)$ is closed under arbitrary coproducts, and also that $\D^{\Tors}_{\gr}(A)$ is a localizing subcategory. To see that $\DQgr(A)$ is locally small, we observe the following. The triangulated category $\D_{\on{gr}}(A)$ arises as the homotopy category of a pretriangulated dg-category $\mathcal{C}$, and $\Tors(A)$ is the homotopy category of a pretriangulated dg-subcategory of $\mathcal{C}$. The local smallness of $\DQgr(A)$ thus follows from the local smallness of $\D_{\on{gr}}(A)$ along with \cite[Proposition 4.7(iii) and (iv)]{drinfeld}. 

With the above observations in hand, we may apply \cite[Example 8.4.5]{neeman} to obtain a right adjoint $\dGamma_* \co \DQgr(A) \to \D_{\on{gr}}(A)$ to $\Pi$. It follows from \cite[Theorem 9.1.16]{neeman} that $\RR\Gamma_*$ is fully faithful. To see this, observe that $\dGamma_*$ is, by definition, a Bousfield localization functor for the inclusion $\D^{\Tors}_{\on{gr}}(A) \into \D_{\on{gr}}(A)$. Letting $^{\perp}\D^{\Tors}_{\on{gr}}(A)$ be as defined in \cite[Definition 9.1.11]{neeman}, the composition of the equivalence $\Dqgr(A) \xra{\simeq} ^{\perp}\D^{\Tors}_{\on{gr}}(A)$ arising from \cite[Theorem 9.1.16]{neeman} with the inclusion $\D^{\Tors}_{\on{gr}}(A) \into \D_{\on{gr}}(A)$ is easily seen to be the right adjoint of $\Pi$ and is clearly fully faithful; by the uniqueness of right adjoints, $\dGamma_*$ is fully faithful as well.

Finally, since $\tau_{\ge n}$ is the right adjoint of the inclusion $\D_{\on{gr}}(A)_{\ge n} \into \D_{\on{gr}}(A)$,  the functor $\dGamma_{\ge n}$ is the right adjoint of the canonical functor $\D_{\on{gr}}(A)_{\ge n} \to \DQgr(A)$. Applying \cite[Theorem 9.1.6]{neeman} and the above argument once again, we conclude that each functor $\dGamma_{\ge n}$ is fully faithful.
\end{proof}

The functor $\dGamma \co \DQgr(A) \to \D_{\on{gr}}(A)$ typically does not  map $\Dqgr(A)$ to $\Dbgr(A)$. For instance, if $A = \k[x_0, x_1]$ with $\bideg(x_0) = \bideg(x_1) = (1,0)$, then $\DQgr(A)$ (resp. $\Dqgr(A)$) is the derived category of quasicoherent (resp. bounded derived category of coherent) sheaves on $\PP^1$, and $\dGamma_*(\OO_{\PP^1})$ is not coherent. However, we now determine conditions under which the functors $\dGamma_{\ge n} \co \DQgr(A) \to \D_{\on{gr}}(A)_{\ge n}$ do map $\Dqgr(A)$ to $\Dbgr(A)_{\ge n}$; this is the content of Proposition~\ref{prop:faithful}. This is well-understood when $A$ is concentrated in cohomological degree 0, due to work of Artin-Zhang~\cite{AZ}, and we will build on their results to prove Proposition~\ref{prop:faithful}. We recall the following definition due to Artin-Zhang:

\begin{defn}[\cite{AZ} Definition 3.7]
\label{def:chi}
Suppose $A$ is concentrated in cohomological degree zero, i.e. $A$ is an ordinary connected $\Z$-graded $\k$-algebra. We say $A$ satisfies \emph{condition $\chi$} if, for every finitely generated graded $A$-module $M$, we have $\dim_\k \underline{\Ext}^i_A(\kk,M) < \infty$ for all $i \in \Z$.\footnote{Definition~\ref{def:chi} is not identical to \cite[Definition 3.7]{AZ}, but it is equivalent: see \cite[Definition 3.2 and Proposition 3.8(1)]{AZ}.} 
\end{defn}

The following lemma follows from results of Artin-Zhang, Yekutieli, and Yekutieli-Zhang:

\begin{lemma}[\cite{AZ, YEKUTIELI1_dualizing, Yekutieli-Zhang}]
\label{lemma:AZ}
Let $A$ be as in Setup~\ref{setup}, and suppose $A$ is concentrated in cohomological degree zero. Let $M$ be a finitely generated graded $A$-module and $n \in \Z$. 
\begin{enumerate}
\item If $A$ satisfies condition $\chi$, and $\RR^j\Gamma_{\ge n}(\w{M}) = 0$ for $j \gg 0$, then $\RR\Gamma_{\ge n}(\w{M}) \in \Dbgr(A)$.
\item If $A$ is either Gorenstein or (strictly) commutative,
then $A$ satisfies condition $\chi$, and $\RR^j\Gamma_*(\w{M}) = \RR^j\Gamma_{\ge n}(\w{M}) = 0$ for $j \gg 0$. In particular, $\RR\Gamma_{\ge n}(\w{M}) \in \Dbgr(A)$ in both cases.
\end{enumerate}
\end{lemma}

\begin{proof}
Part (1) follows from \cite[Theorem 7.4]{AZ}.
As for (2): if $A$ is commutative, then condition $\chi$ holds by \cite[Proposition 3.11(3)]{AZ}, and $\RR^j\Gamma_*(\w{M}) = H^j(\Proj(A), \w{M}) =~0$ for $j \gg 0$ since $\Proj(A)$ is projective over~$\k$. It follows that, given $n \in \Z$, we have $\RR^j\Gamma_{\ge n}(\w{M}) = 0$ for all $j \gg 0$ as well. If $A$ is Gorenstein, then $A$ has a balanced dualizing complex (in the sense of \cite[Definition 4.1]{YEKUTIELI1_dualizing}) by \cite[Section 4]{YEKUTIELI1_dualizing}, and so the result follows from \cite[Theorem 4.2(3)]{Yekutieli-Zhang}. 
\end{proof}

\begin{prop}
\label{prop:faithful}
Let $A$ be as in Setup~\ref{setup}. If $A^0$ is Gorenstein or (strictly) commutative, then 
the canonical functor 
$\pi_n \co \Dbgr(A)_{\ge n} \to \Dqgr(A)$ admits a fully faithful right adjoint given by $\dGamma_{\ge n}$ for all $n \in \Z$. 
\end{prop}

\begin{remark}
\label{rem:bidual1}
The assumption in Proposition~\ref{prop:faithful} that $A^0$ is Gorenstein or strictly commutative is needed to apply Lemma~\ref{lemma:AZ}. We note that these assumptions can be relaxed: one only needs $A^0$ to admit a balanced dualizing complex (in the sense of \cite[Definition 4.1]{YEKUTIELI1_dualizing}) for the results of Lemma~\ref{lemma:AZ} to hold. For instance, condition $\chi$ follows in this case from \cite[Theorem 4.2(3)]{Yekutieli-Zhang}.
\end{remark}

\begin{proof}
First, we show that the derived extension of scalars functor $F \co \D_{\on{gr}}(A^0) \to \D_{\on{gr}}(A)$ given by $M \mapsto M \otimes^\LL_{A^0} A$ sends $\D^{\Tors}_{\on{gr}}(A^0)$ to $\D^{\Tors}_{\on{gr}}(A)$. Since every object in $\D^{\Tors}_{\on{gr}}(A^0)$ is a filtered colimit of objects in $\D^{\tors}_{\on{gr}}(A^0)$, $F$ commutes with colimits, and $\D^{\Tors}_{\on{gr}}(A)$ is closed under filtered colimits; it suffices to show that $F$ maps $\D^{\tors}_{\on{gr}}(A^0)$ to $\D^{\Tors}_{\on{gr}}(A)$. Certainly $F(\kk(j)) \in \D^{\Tors}_{\on{gr}}(A)$ for all $j \in \Z$; by Lemma~\ref{thick}, it follows that $F(\D^{\tors}_{\on{gr}}(A^0)) \subseteq \D^{\Tors}_{\on{gr}}(A)$. 

It is clear that the restriction of scalars functor $G \co \D_{\on{gr}}(A) \to \D_{\on{gr}}(A^0)$ sends $\D^{\Tors}_{\on{gr}}(A)$ to $\D^{\Tors}_{\on{gr}}(A^0)$. By \cite[Lemma 1.1]{Orlov2009}, we conclude that the adjunction $F \co \D_{\on{gr}}(A^0) \rightleftarrows \D_{\on{gr}}(A) \co G$ induces an adjunction $\ov{F} \co \DQgr(A^0) \rightleftarrows \DQgr(A) \co \ov{G}$, and moreover we have a commutative square
\begin{equation}
\label{square}
\xymatrix{
\D_{\on{gr}}(A^0) \ar[d]^-{F} \ar[r]^-{\pi} & \DQgr(A^0) \ar[d]^-{\ov{F}} \\
\D_{\on{gr}}(A) \ar[r]^-{\pi} & \DQgr(A). \\
}
\end{equation}
Replacing the left adjoints in \eqref{square} with their right adjoints, we obtain the commutative square

\begin{equation}
\label{squarekey}
\xymatrix{
\D_{\on{gr}}(A^0)   & \ar[l]_-{\dGamma_*} \DQgr(A^0)  \\
\D_{\on{gr}}(A) \ar[u]_-{G}  & \ar[l]_-{\dGamma_*} \DQgr(A). \ar[u]_-{\ov{G}} \\
}
\end{equation}
Let $n \in \Z$. Concatenating \eqref{squarekey} with the commutative diagram
$$
\xymatrix{
\D_{\on{gr}}(A^0)_{\ge n}  & \ar[l]_-{\tau_{\ge n}} \D_{\on{gr}}(A^0)  \\
\D_{\on{gr}}(A)_{\ge n} \ar[u]_-{G}  & \ar[l]_-{\tau_{\ge n}} \D_{\on{gr}}(A), \ar[u]_-{G} \\
}
$$
we arrive at the key commutative square
\begin{equation}
\label{squarekey2}
\xymatrix{
\D_{\on{gr}}(A^0)_{\ge n}   & \ar[l]_-{\dGamma_{\ge n}} \DQgr(A^0)  \\
\D_{\on{gr}}(A)_{\ge n}  \ar[u]_-{G}  & \ar[l]_-{\dGamma_{\ge n}} \DQgr(A). \ar[u]_-{\ov{G}} \\
}
\end{equation}
Let $\w{M} \in \Dqgr(A)$. We have $\ov{G}(\w{M}) \in \Dqgr(A^0)$. By our assumptions on $A^0$, Lemma~\ref{lemma:AZ}(2) implies that $\dGamma_{\ge n}(\ov{G}(\w{M})) \in \Dbgr(A^0)_{\ge n}$. By the commutativity of \eqref{squarekey2}, it follows that $G(\dGamma_{\ge n}(\w{M})) \in \Dbgr(A^0)_{\ge n}$, and this implies $\dGamma_{\ge n}(\w{M}) \in \Dbgr(A)_{\ge n}$. Thus, $\dGamma_{\ge n}$ is the right adjoint of $\pi_n$.
\end{proof}

We record the following observation, which follows from the proof of Proposition~\ref{prop:faithful}:

\begin{prop}
\label{prop:explicit}
Let $A$ be as in Setup~\ref{setup} and $M$ a dg-$A$-module. Denote by $G : \Dbgr(A) \to \Dbgr(A^0)$ the restriction of scalars functor. There is an isomorphism
$$
G(\dGamma_*(\w{M})) \cong \underset{p \to \infty}{\colim}_{\text{ }}\RHom_{A^0}(A^0_{\ge p}, G(M)).
$$
\end{prop}

\begin{proof}
Let $\ov{G} \co \DQgr(A) \to \DQgr(A^0)$ denote the functor induced by restriction of scalars. By the commutativity of the square~\eqref{squarekey}, it suffices to show
\begin{equation}
\label{eqn:a0}
\RR\Gamma_*(\ov{G}(\w{M})) \cong \underset{p \to \infty}{\colim}_{\text{ }}\RHom_{A^0}(A^0_{\ge p}, G(M)).
\end{equation}
Noting that, by \cite[Lemma 4.4.1]{krause}, $\Dqgr(A^0)$ coincides with the derived category of the abelian quotient $\qgr(A^0)$ of finitely generated $A^0$-modules by torsion modules, the isomorphism \eqref{eqn:a0} follows from (the proof of) \cite[Lemma 4.1]{Yekutieli-Zhang}, along with the adjunction isomorphism $\dGamma_*(\w{N})  \cong \bigoplus_{i, j \in \Z} \Hom_{\Dqgr(A^0)}(\widetilde{A^0}, \widetilde{N(i)[j]})$ for $N \in \Dbgr(A^0)$. 
\end{proof}

\subsection{Sheaves of dg-algebras}

In certain settings, the $\DQgr$ and $\Dqgr$ constructions for dg-algebras may be interpreted as derived categories of sheaves. Our next goal is to make this precise.

\begin{defn}
Let $Y$ be a scheme. A \emph{dg-$\OO_Y$-algebra} is a complex $\B$ of $\OO_Y$-modules such that $\Gamma(U, \B)$ is a differential $\Z$-graded $\Gamma(U, \OO_Y)$-algebra for all open sets $U$ in $Y$, and the restriction maps for $\B$ are morphisms of differential $\Z$-graded $\Gamma(Y, \OO_Y)$-algebras. Given a dg-$\OO_Y$-algebra $\B$, a \emph{dg-$\B$-module} is a complex $\C$ of $\OO_Y$-modules such that $\Gamma(U, \C)$ is a differential $\Z$-graded $\Gamma(U,\B)$-module for all open sets $U$ in $Y$, and the restriction maps for $\C$ are morphisms of differential $\Z$-graded $\Gamma(Y, \B)$-modules. A dg-$\B$-module $\C$ is \emph{quasi-coherent} (resp. \emph{coherent}) if $\bigoplus_{i \in \Z} \C^i$ is quasi-coherent (resp. coherent) as an $\OO_Y$-module. Morphisms of dg-$\B$-modules are defined in the evident way. Let $\QCoh(\B)$ (resp. $\Coh(\B)$) denote the category of quasi-coherent (resp. coherent) dg-$\B$-modules, 
 and let $\D(\QCoh{\B})$ and $\D(\Coh{\B})$ denote their derived categories: see \cite[\href{https://stacks.math.columbia.edu/tag/0FT1}{Tag 0FT1}]{stacks-project} for details.
 Given $\mathcal{M}, \mathcal{N} \in \QCoh(\B)$, the Hom sheaf $\mathcal{H}\text{om}_{\B}(\mathcal{M}, \mathcal{N})$ is defined in the evident way: see  \cite[\href{https://stacks.math.columbia.edu/tag/0FRN}{Tag 0FRN}]{stacks-project}. Note that $\mathcal{H}\text{om}_{\B}(\mathcal{M}, \mathcal{N})$ is a complex of $\OO_Y$-modules but need not be a dg-$\B$-module. We have $\Hom_{\QCoh(\B)}(\mathcal{M}, \mathcal{N}) \cong Z^0\Gamma(\mathcal{H}\text{om}_{\B}(\mathcal{M}, \mathcal{N}))$ as $\k$-vector spaces. 
\end{defn}

We adopt the following setup in this subsection:

\begin{setup}
\label{setup2}
Let $A$ be as in Setup~\ref{setup}. Assume $A^0$ is (strictly) commutative, Noetherian, and generated in degree 1; and suppose $A$ is finitely generated as an $A^0$-module. Let $X$ denote the projective $\k$-scheme $\Proj(A^0)$. We assume also that $A$ is \emph{saturated} as an $A^0$-module, meaning that the natural map $A \to \bigoplus_{i \in \Z} H^0(X, \w{A}(i))$ is an isomorphism. Equivalently, letting $\mathfrak{n}$ denote the homogeneous maximal ideal of $A^0$, we assume $H^0_{\mathfrak{n}}(A) = H^1_{\mathfrak{n}}(A) = 0$.
\end{setup}

\begin{example}
Suppose $R$ is a standard graded polynomial ring over $\k$ in at least 2 variables, and denote its homogeneous maximal ideal by $\mathfrak{n}$. If $A$ is the Koszul complex on a (not necessarily regular) sequence of homogeneous elements in $R$, or if $A$ is the dg-algebra $F \otimes_R F$ from \Cref{Ex:gorenstein_tensor_product}, then $A$ satisfies the conditions in Setup~\ref{setup2}. For instance, $A$ is a saturated $R$-module since it is free over $R$, and $H^0_{\mathfrak{n}}(R) = H^1_{\mathfrak{n}}(R) = 0$.
\end{example}

Let $A$ be as in Setup~\ref{setup2}. Recall from  Remark~\ref{A0} that the differential on every dg-$A$-module is $A^0$-linear. We have a dg-$\OO_X$-algebra $\A$ given by the sheafification of the complex $A$ of graded $A^0$-modules, with $\OO_X$-linear differential induced by $d$. There is a functor $\Sh \co \Mod(A) \to \QCoh(\A)$, where $\Sh(M, d_M)$ is the sheafification of the complex $M$ of graded $A^0$-modules with $\OO_X$-linear differential induced by $d_M$.  We also have a global sections functor $\Gamma_* \co \Qcoh(\A) \to \Mod(A)$ that sends an object $\mathcal{M} \in \Qcoh(\A)$ to the complex of $A^0$-modules $\Gamma_*(\mathcal{M}) = \bigoplus_{i \in \Z} H^0(X, \mathcal{M}(i))$ equipped with the evident dg-$A$-module structure.

The functor $\Sh$ is exact and therefore induces a triangulated functor on derived categories $\Sh \co \D_{\on{gr}}(A) \to \D(\Qcoh{\A})$. On the other hand, the functor $\Gamma_* \co \Qcoh(\A) \to \Mod(A)$ induces a right derived functor $\RR\Gamma_* \co \D(\Qcoh{\A}) \to \D_{\on{gr}}(A)$ (the overlap of notation between $\RR\Gamma_*$ in this section and the derived global sections functors on $\Dqgr(A)$ defined in the previous section should not cause confusion and is justified by Corollary~\ref{cor:sections} below). Let us give an explicit construction of $\RR\Gamma_*$, following \cite[\href{https://stacks.math.columbia.edu/tag/0FTN}{Tag 0FTN}]{stacks-project} and \cite[\href{https://stacks.math.columbia.edu/tag/0FTP}{Tag 0FTP}]{stacks-project}(2). 
A dg-$\A$-module $\I$ is said to be \textit{K-injective} if, for every exact dg-$\A$-module $\M$, the complex $\Hom_{\Qcoh(\A)}(\M,\I)$ of $\k$-vector spaces is exact. 
Let $\mathcal{M} \in \Qcoh(\A)$, and choose a quasi-isomorphism $\M\xra{\simeq} \I$ of dg-$\A$-modules such that $\I$ is an injective object in the category of graded $\A$-modules (i.e. $\I$ is \emph{graded injective}, in the sense of \cite[\href{https://stacks.math.columbia.edu/tag/0FSP}{Tag 0FSP}]{stacks-project}) and also $K$-injective; such resolutions always exist, and the choices can be made functorially \cite[\href{https://stacks.math.columbia.edu/tag/0FT0}{Tag 0FT0}]{stacks-project}. We define $\RR \Gamma_*(\mathcal{M}) \ce \Gamma_*(\I)$, and we set $\RR \Gamma_{\ge n} \ce \tau_{\ge n} \RR\Gamma$.

\begin{lemma}
\label{lem:adjointsheaf}
Let $A$ be as in Setup~\ref{setup2}.
\begin{enumerate}
\item The functor $\Gamma_* \co \Qcoh(\A) \to \Mod(A)$ is the right adjoint of $\Sh$. In fact, there is a natural isomorphism of complexes of graded $A^0$-modules $\Gamma_*(\mathcal{H}\emph{om}_{\A}(\Sh(M), \N)) \cong \uHom_A(M, \Gamma_*(\N))$ for all dg-$A$-modules $M$ and $\N \in \Qcoh(\A)$.
\item The functor $\RR \Gamma_*$ is the right adjoint of $\Sh \co \D_{\on{gr}}(A) \to \D(\Qcoh{\A})$. Consequently, for all $n \in \Z$, the functor $\RR\Gamma_{\ge n}$ is the right adjoint of $\D_{\on{gr}}(A)_{\ge n} \into \D_{\on{gr}}(A) \xra{\Sh} \D(\Qcoh{\A})$.
\end{enumerate}
\end{lemma}

\begin{proof}
Let $M \in \Mod(A)$ and $\N \in \Qcoh(\A)$. The first statement in (1) follows immediately from the isomorphism $\Gamma_*(\mathcal{H}\emph{om}_{\A}(\Sh(M), \N)) \cong \uHom_A(M, \Gamma_*(\N))$ by passing to bidegree (0,0) cocycles, so, to prove (1), we need only consider the second statement. 
This is well-known over $A^0$, and so there is an isomorphism of complexes of graded $A^0$-modules 
\begin{equation}
\label{eqn:X}
\Gamma_* (\mathcal{H}\text{om}_{X}(\Sh(M), \N)) \cong \uHom_{A^0}(M, \Gamma_*(\N)).
\end{equation}
It is routine to check that \eqref{eqn:X} restricts to the desired isomorphism. 
For (2), choose a quasi-isomorphism $\N \xra{\simeq} \I$, where $\I$ is an injective object in the category of graded $\A$-modules  and also $K$-injective. Choose also  a semi-free resolution $F \xra{\simeq} M$. Let $K(A)$ (resp. $K(\A)$) denote the homotopy category of dg-$A$-modules (resp. dg-$\A$-modules). We note that $\Hom_{\D_{\on{gr}}(A)}(M, L) = \Hom_{K(A)}(F, L)$ for all $L \in \Mod(A)$, and $\Hom_{\D(\Qcoh{\A})}(\mathcal{G}, \N) = \Hom_{K(\A)}(\mathcal{G}, \mathcal{I})$ for all $\mathcal{G} \in \D(\Qcoh{\A})$; see \cite[Theorem 10.1.13]{Yekutieli2020} and \cite[\href{https://stacks.math.columbia.edu/tag/0FT8}{Tag 0FT8}]{stacks-project}. We now compute:
\begin{align*}
\Hom_{\D_{\on{gr}}(A)}(M, \RR \Gamma_*(\N)) &= \Hom_{\D_{\on{gr}}(A)}(M, \Gamma_*(\I)) \\
&\cong \Hom_{K(A)}(F, \Gamma_*(\I)) \\
& = H^0 \uHom_A(F, \Gamma_*(\I))_0 \\
& \cong H^0 \Gamma_*(\mathcal{H}\text{om}_{\A}(\Sh(F), \I))_0 \\
& = \Hom_{K(\A)}(\Sh(F), \I) \\
& \cong \Hom_{\D(\Qcoh{\A})}(\Sh(M), \N).
\end{align*} 
This proves the first statement of (2). The second statement follows immediately from the first, since $\tau_{\ge n}$ is the right adjoint of the inclusion $\D_{\on{gr}}(A)_{\ge n} \into \D_{\on{gr}}(A)$. 
\end{proof}

The functor $\Sh$ induces a functor $\ov{\Sh} \co \DQgr(A) \to \D(\Qcoh{\A})$.
On the other hand, composing $\RR\Gamma_* \co \D(\Qcoh{\A}) \to \D_{\on{gr}}(A)$ with the projection functor $\Pi: \D_{\on{gr}}(A)\to \DQgr(A)$, one obtains a triangulated functor $T: \D(\Qcoh{\A}) \to \DQgr(A)$.

\begin{thm}
\label{prop:sheaf}
Let $A$ be as in Setup~\ref{setup2}. 
\begin{enumerate}
\item The functor $\ov{\Sh} \co \DQgr(A) \to \D(\Qcoh{\A})$ is the left adjoint of $T$. 
\item The functors $\ov{\Sh}$ and $T$ are inverse equivalences. 
\item We have $\ov{\Sh}(\Dqgr(A)) \subseteq \D(\coh{\A})$, and $T(\D(\coh{\A})) \subseteq \Dqgr(A).$ Consequently, the adjunction in (1) restricts to an adjunction $\ov{\Sh} \co \Dqgr(A) \leftrightarrows \D(\coh{\A}) \co T$, and these functors are also inverse equivalences.
\end{enumerate}
\end{thm}

\begin{proof} 
(1) follows from Lemma~\ref{lem:adjointsheaf}(2) and \cite[Lemma 1.1]{Orlov2009}. Let $\mathcal{M} \in \D(\Qcoh{\A})$. Choose a quasi-isomorphism $\mathcal{M} \xra{\simeq} \I$, where $\I$ is an injective object in the category of graded $\A$-modules and also $K$-injective. The canonical map $\Sh\Gamma_*(\I) \to \I$ is an isomorphism of $\OO_X$-modules by \cite[Proposition II.5.15]{hartshorne}, and one checks that it is a morphism of $\A$-modules. We therefore have an isomorphism $\mathcal{M} \cong \Sh \RR\Gamma_*(\mathcal{M})$ 
in $\D(\Qcoh{\A})$; in particular, $\ov{\Sh}$ is essentially surjective. To prove $\ov{\Sh}$ is fully faithful, we show that the unit of the adjunction $\ov{\Sh} \co \DQgr(A) \rightleftarrows \D(\Qcoh{\A}) \co T$ is an isomorphism. This amounts to the assertion that, given $M \in \Mod(A)$, the natural map $M \to \RR\Gamma_*(\w{M})$ induces an isomorphism on cohomology up to torsion. This statement only concerns the underlying structure of complexes of $A^0$-modules. Since it is well-known that $\RR\Gamma_*$ determines an equivalence $\D(\coh{X}) \xra{\simeq} \Dqgr(A^0)$, it follows from \cite[\href{https://stacks.math.columbia.edu/tag/0FTW}{Tag 0FTW}]{stacks-project} that the unit of this adjunction must be an isomorphism; this proves (2). It is clear that $\ov{\Sh}(\Dqgr(A)) \subseteq \D(\coh{\A})$. Going the other direction: once again by \cite[\href{https://stacks.math.columbia.edu/tag/0FTW}{Tag 0FTW}]{stacks-project}, we have a commutative diagram
$$
\xymatrix{
\D(\Qcoh{\A}) \ar[d] \ar[r]^-{\dGamma_{\ge 0}} & \D_{\on{gr}}(A)_{\ge 0} \ar[d] \ar[r] & \DQgr(A) \ar[d]\\
\D(\Qcoh{X}) \ar[r]^-{\dGamma_{\ge 0}} & \D_{\on{gr}}(A^0)_{\ge 0} \ar[r] & \DQgr(A^0),
}
$$
where the vertical maps are forgetful functors, and the right-most horizontal maps are the canonical ones. Observe that composing the two maps along the top row gives the functor $T$; it therefore suffices to show that $\dGamma_{\ge 0}$ maps $\D(\coh{\A})$ to $\Dbgr(A)_{\ge 0}$. It is well-known that the map $\dGamma_{\ge 0} \co \D(\Qcoh{X}) \to \D_{\on{gr}}(A^0)_{\ge 0}$ maps $\Db(\coh{X})$ to $\Dbgr(A^0)_{\ge 0}$, and an object in $\D_{\on{gr}}(A)_{\ge 0}$ is contained in $\Dbgr(A^0)_{\ge 0}$ if and only if it is contained in $\Dbgr(A)_{\ge 0}$. Thus, the commutativity of the diagram implies (3).
\end{proof}

\begin{cor}
\label{cor:sections}
Let $A$ be as in Setup~\ref{setup2}. For all $i \in \Z$, we have commutative triangles
$$
\xymatrix{
\D(\Qcoh{\A}) \ar[d]^{\simeq}_-{T} \ar[r]^-{\RR\Gamma_*} & \D_{\on{gr}}(A)  \\
\DQgr(A), \ar[ru]_-{\RR\Gamma_*} 
}
\quad \quad
\xymatrix{
\D(\Qcoh{\A}) \ar[d]^{\simeq}_-{T} \ar[r]^-{\RR\Gamma_{\ge i}} & \D_{\on{gr}}(A)_{\ge i}  \\
\DQgr(A), \ar[ru]_-{\RR\Gamma_{\ge i}} 
}
\quad \quad
\xymatrix{
\D(\coh{\A}) \ar[d]^{\simeq}_-{T} \ar[r]^-{\RR\Gamma_{\ge i}} & \Dbgr(A)_{\ge i} \\
\Dqgr(A). \ar[ru]_-{\RR\Gamma_{\ge i}} 
}
$$
That is, our versions of derived global section functors may be identified, in this setting, via the equivalence $T$ from Theorem~\ref{prop:sheaf}(2).
\end{cor}

\begin{proof}
This follows from Propositions~\ref{prop:rightadjoint} and~\ref{prop:faithful}, Lemma~\ref{lem:adjointsheaf}(2), Theorem~\ref{prop:sheaf}, and the uniqueness of right adjoints. 
\end{proof}

\section{Proof of Theorem~\ref{main}}
\label{sec:proof}

The proof of our main result, Theorem~\ref{main}, goes roughly as follows. Let $A$ be as in Setup~\ref{setup}, where $A^0$ is Gorenstein or (strictly) commutative, and fix $i \in \Z$.  Proposition~\ref{prop:faithful} implies that the truncated derived global section functor $\dGamma_{\ge i}$ identifies $\Dqgr(A)$ with a weak semiorthogonal summand $\mathcal{D}_i$ of $\Dbgr(A)_{\ge i}$. Recall from \Cref{def:semi-free} that $\Perf(A)$ denotes the thick subcategory of $\Dbgr(A)$ given by perfect objects; let $\Dsg(A)$ denote the quotient $\Dbgr(A) / \Perf(A)$, the \emph{singularity category of $A$}. Assume in addition that $A$ is Gorenstein, and let $a$ be its Gorenstein parameter; Lemma~\ref{lemma:sod2} below yields an embedding of $\Dsg(A)$ as a semiorthogonal summand $\mathcal{T}_i$ of $\Dbgr(A)_{\ge i}$. Our proof of Theorem~\ref{main}, which mirrors Orlov's original argument, shows that $\mathcal{T}_i$ is in fact a semiorthogonal summand of $\mathcal{D}_i$ when $a \ge 0$, and $\mathcal{D}_i$ is a semiorthogonal summand of $\mathcal{T}_i$ when $a \le 0$; and moreover the complements in each case are given by exceptional collections.
\medskip

Let us fix some notation. For $i \in \Z$, let $\mathcal{P}_{<i}$ (resp. $\mathcal{P}_{\geq i}$) denote the thick subcategory of $\Perf(A)$ generated by the modules $A(e)$ for $e>-i$ (resp. $A(e)$ for $e\leq -i$), and let $\mathcal{S}_{<i}$ (resp. $\mathcal{S}_{\geq i}$) denote the thick subcategory of $\Dbgr(A)$ generated by the modules $\k(e)$ for $e>-i$ (resp. $\k(e)$ for $e\leq -i$). Define the subcategories $\mathcal{P}^{\op}_{<i}$, $\mathcal{P}^{\op}_{\geq i}$ of $\Perf(A^{\op})$ and $\S^{\op}_{< i}$, $\S^{\op}_{\ge i}$ of $\Dbgr(A^{\op})$ similarly. 

\medskip
We now recall some background on semiorthogonal decompositions, following \cite[Section 1]{Orlov2009}. 
Given a $\k$-linear triangulated category $\mathcal{B}$ and a full triangulated subcategory $\mathcal{C}$ of $\mathcal{B}$, the \emph{right orthogonal} $\mathcal{C}^{\perp}$ of $\mathcal{C}$ is the triangulated subcategory given by  $\{B \in \mathcal{B} \text{ : } \Hom_{\mathcal{B}}(C, B) = 0 \text{ for all } C \in \mathcal{C} \}$, and the \emph{left orthogonal} $^\perp\mathcal{C}$ is defined similarly. We say $\mathcal{C}$ is \emph{right admissible} (resp. \emph{left admissible}) if the inclusion $\mathcal{C} \into \mathcal{B}$ admits a right (resp. left) adjoint, and we say $\mathcal{C}$ is \emph{admissible} if it is left and right admissible. The subcategory $\mathcal{C}$ is right (resp. left) admissible if and only if for all $B \in \mathcal{B}$, there is an exact triangle
$$
B' \to B \to B'' \to B'[1]
$$
such that $B' \in \mathcal{C}$ and $B'' \in \mathcal{C}^\perp$ (resp. $B' \in$ $^\perp\mathcal{C}$ and $B'' \in \mathcal{C}$). A sequence of triangulated subcategories $\mathcal{C}_1, \dots, \mathcal{C}_n$ of $\mathcal{B}$ forms a \emph{weak semiorthogonal decomposition} (resp. \emph{semiorthogonal decomposition}) of $\mathcal{B}$ if there are left admissible (resp. admissible) subcategories
$$
\mathcal{B}_1 = \mathcal{C}_1 \subseteq \mathcal{B}_2 \subseteq \cdots \subseteq \mathcal{B}_n = \mathcal{B}
$$
such that each $\mathcal{C}_i$ is the left orthogonal of $\mathcal{B}_{i-1}$ in $\mathcal{B}_{i}$. When $\mathcal{C}_1, \dots, \mathcal{C}_n$ form a weak semiorthogonal decomposition of $\mathcal{B}$, we write $\mathcal{B} = \langle \mathcal{C}_1, \dots, \mathcal{C}_n \rangle$. 

An object $B$ in $\mathcal{B}$ is called \emph{exceptional} if $\Hom_{\mathcal{B}}(B, B[j]) = 0$ for $j \ne 0$, and $\Hom_{\mathcal{B}}(B, B) \cong\k$. An \emph{exceptional collection} in $\mathcal{B}$ is a sequence $E_1, \dots, E_n$ of exceptional objects such that $\Hom_{\mathcal{B}}(E_i, E_j[\ell]) = 0$ for all $\ell \in \Z$ when $i > j$. An exceptional collection $E_1, \dots, E_n$ is called \emph{full} if the objects $E_1, \dots, E_n$ generate all of $\mathcal{B}$. When $E_1, \dots, E_n$ form a full exceptional collection, the subcategories $\mathcal{C}_i$ of $\mathcal{B}$ generated by the $E_i$ form a semiorthogonal decomposition of $\mathcal{B}$; in this case, we write $\mathcal{B} = \langle E_0, \dots, E_n \rangle$. An exceptional collection $E_0, \dots, E_n$ is called \emph{strong} if $\Hom_{\mathcal{B}}(E_i, E_j[\ell]) = 0$ for all $i$ and $j$ when $\ell \ne 0$.

\begin{lemma}[cf. \cite{Orlov2009} Lemma 2.3]
\label{lemma:sod1}
Fix $i \in \Z$. The subcategory $\S_{<i}$ (resp. $\P_{< i}$) of $\Dbgr(A)$ is left (resp. right) admissible, and we have weak semiorthogonal decompositions
$$
\Dbgr(A)  = \langle \S_{< i}, \Dbgr(A)_{\ge i} \rangle = 
\langle \Dbgr(A)_{\ge i}, \P_{<i} \rangle, \text{ }
\on{D}^{\tors}_{\on{gr}}(A) = \langle \S_{< i}, \S_{\ge i} \rangle, \text{ }
\Perf(A)  = \langle \P_{\ge i}, \P_{< i} \rangle.
$$
\end{lemma}
\begin{proof}
Given $M \in \Dbgr(A)$, we have a short exact sequence $0 \to M_{\ge i} \to M \to M / M_{\ge i} \to 0$. The object $M / M_{\ge i}$ is in $\S_{< i}$, and $M_{\ge i}$ is in $^{\perp} \S_{<i}$. It therefore follows from \cite[Remark 1.3]{Orlov2009} that $\S_{< i}$ is left admissible. Moreover, $\Dbgr(A)_{\ge i}$ is the left orthogonal $^\perp\S_{< i}$ of $\S_{< i}$ in $\Dbgr(A)$, i.e. we have $\Dbgr(A)  = \langle \S_{< i}, \Dbgr(A)_{\ge i} \rangle$. If $M$ is torsion, then $M_{\ge i} \in \S_{\ge i}$, and so we conclude $\on{D}_{\tors}(A) = \langle \S_{< i}, \S_{\ge i} \rangle$.

Let $M \in \Dbgr(A)$, and choose a semi-free resolution $F$ of $M$ as in Proposition~\ref{prop:semi-free}. Let $F'$ be the dg-submodule of $F$ consisting of free summands generated in internal degree strictly less than $i$. We have a short exact sequence $0 \to F' \to F \to F/F' \to 0$; observe that $F/F' \in \Dbgr(A)_{\ge i}$, and the properties of $F$ guaranteed by Proposition~\ref{prop:semi-free} imply that $F'$ is contained in $\Perf(A)$ and hence in $\P_{< i}$. This yields the remaining two semiorthogonal decompositions.
\end{proof}

\begin{lem}[cf. \cite{Orlov2009} Lemma 2.4]
    \label{lemma:sod2}
    Let $A$ be as in Setup~\ref{setup}. Assume $A$ is Gorenstein and that either $A^0$ is Gorenstein or (strictly) commutative. Let $a$ be the Gorenstein parameter of $A$, and fix $i\in \mathbb{Z}$. The subcategory $\mathcal{S}_{\geq i}$ (resp. $\mathcal{P}_{\geq i}$) is right (resp. left) admissible in $\Dbgr(A)_{\geq i}$, and there are weak semiorthogonal decompositions
    $$
    \Dbgr(A)_{\geq i}=\langle \mathcal{D}_i, \mathcal{S}_{\geq i}\rangle = \langle \mathcal{P}_{\geq i},\mathcal{T}_i\rangle,
    $$
    where $\mathcal{D}_i$ is the essential image of $\dGamma_{\ge i}$, and the composition $\mathcal{T}_i \into \Dbgr(A)_{\ge i} \xra{q} \Dsing(A)$ is an equivalence. Moreover, the right orthogonals of the subcategories $\mathcal{D}_i$ and $\mathcal{T}_i$ of $\Dbgr(A)$ are as follows:
    $$
    \mathcal{T}_i^{\perp} =  \langle  \mathcal{S}_{<i},\mathcal{P}_{\geq i}\rangle \quad \text{and} \quad
    \mathcal{D}_i^{\perp} = \langle \mathcal{P}_{\geq i+a}, \mathcal{S}_{<i}\rangle.
    $$
\end{lem}

\begin{proof}
    By Proposition~\ref{prop:faithful}, $\dGamma_{\ge i} \co \Dqgr(A) \to \mathcal{D}_i$ is an equivalence. Since $\dGamma_{\ge i}$ is the right adjoint of the canonical functor $\Dbgr(A)_{\ge i} \to \Dqgr(A)$,
    it follows that $\mathcal{D}_i$ is left admissible in $\Dbgr(A)_{\ge i}$, with left orthogonal $\mathcal{S}_{\geq i}$. We conclude that $\mathcal{S}_{\geq i}$ is right admissible in $\Dbgr(A)_{\geq i}$, and there is a weak semiorthogonal decomposition $\Dbgr(A)_{\geq i}=\langle \mathcal{D}_i, \mathcal{S}_{\geq i}\rangle$.
     
    Since $A$ is Gorenstein, the functor $\RHom_{A^{\op}}( -, A)$ gives an equivalence
    \begin{equation}
    \label{eqn:duality}
    \Dbgr(A^{\op}) \xra{\simeq} \Dbgr(A)^{\op}
    \end{equation}
    By Lemma~\ref{lemma:sod1}, the subcategory $\mathcal{P}^{\op}_{< -i + 1}$ of $\Dbgr(A^{\op})$ is right admissible; by duality, it follows that $\mathcal{P}_{\ge i}$ is left admissible in $\Dbgr(A)$, and hence in $\Dbgr(A)_{\ge i}$ as well. We thus have a weak semiorthogonal decomposition $\Dbgr(A)_{\ge i} = \langle \mathcal{P}_{\ge i}, \mathcal{T}_i\rangle$ for some $\mathcal{T}_i \subseteq \Dbgr(A)_{\ge i}$. By \cite[Lemma 1.4]{Orlov2009}, we have $\mathcal{T}_i \simeq \Dbgr(A)_{\geq i}/\mathcal{P}_{\geq i}$. It therefore follows from \cite[Lemma 1.1]{Orlov2009} that the canonical functor $\Dbgr(A)_{\geq i}/\mathcal{P}_{\geq i}\rightarrow \Dbgr(A)/\Perf(A)=\Dsing(A)$ is fully faithful. By \cref{prop:semi-free}, this functor is also essentially surjective and hence an equivalence. Finally, we observe that the composite functor $\mathcal{T}_i\to \Dsing(A)$ coincides with the composition $\mathcal{T}_i \into \Dbgr(A)_{\ge i} \xra{q}  \Dsg(A)$.
    
    The equality $\mathcal{T}_i^{\perp} = \langle \mathcal{S}_{<i},\mathcal{P}_{\geq i}\rangle$ is immediate from Lemma~\ref{lemma:sod1}. We now show $\mathcal{D}_i^{\perp} = \langle \mathcal{P}_{\ge i + a}, \mathcal{S}_{< i} \rangle.$ Using Lemma~\ref{lemma:sod1} and the equality $\Dbgr(A)_{\ge i} = \langle \mathcal{D}_i, \S_{\ge i} \rangle$, we have $\langle \S_{< i}, \mathcal{D}_i \rangle = \S_{\ge i}^\perp$. The Gorenstein condition on $A$ implies that $\mathcal{S}_{\geq i}^\perp \subseteq \Dbgr(A)$ corresponds to $^{\perp}\mathcal{S}_{<-i-a+1}^{\op} \subseteq \Dbgr(A^{\op})$ via the duality~\eqref{eqn:duality}. The subcategory $^{\perp}\mathcal{S}_{<-i-a+1}^{\op} \subseteq \Dbgr(A^{\op})$ coincides with $(\mathcal{P}_{<-i-a+1}^{\op})^{\perp}\subseteq \Dbgr(A^{\op})$ by Lemma~\ref{lemma:sod1}. Finally, applying the duality~\eqref{eqn:duality} to $(\mathcal{P}_{<-i-a+1}^{\op})^{\perp}$ gives $^{\perp}\mathcal{P}_{\geq i+a} \subseteq \Dbgr(A)$; we conclude that $\langle \S_{< i}, \mathcal{D}_i \rangle$ is equal to $^{\perp}\mathcal{P}_{\geq i+a}$, i.e. $\mathcal{D}_i^{\perp} = \langle \mathcal{P}_{\ge i + a}, \mathcal{S}_{< i} \rangle.$
\end{proof}

\begin{lem}
\label{lem:exceptional}
Let $A$ be as in \Cref{setup}. 
We have full exceptional collections
$$
\Perf(A) = \langle \dots, A(-1), A, A(1), \dots \rangle \quad \text{and} \quad
\D^{\tors}_{\on{gr}}(A) = \langle \dots, \k(1), \k, \k(-1), \dots \rangle.
$$
\end{lem}

\begin{proof}
Fix $j,m,p\in \mathbb{Z}$. The statement concerning $\Perf(A)$ follows from~\Cref{rem:perf} and the evident isomorphism
$
\Hom_{\D_{\on{gr}}(A)}(A(m),A(j)[p]) \cong H(A)^p_{j-m}.
$
Similarly, the statement about $\D^{\tors}_{\on{gr}}(A)$ follows from \Cref{thick} and the isomorphism $\Hom_{\D_{\on{gr}}(A)}(\kk(m),\kk(j)[p]) \cong H(\kk)^p_{j-m}$ whenever $j \ge m$; let us prove this isomorphism. Let $G \xra{\simeq} A(m)_{\ge -m + 1}$ be a semi-free resolution as in \Cref{prop:semi-free}, and observe that, if $j \ge m$, then 
$
\Hom_{\D_{\gr}(A)}(A(m)_{\geq -m+1},\kk(j)[p]) = \Hom_A(G, \k(j)[p]) = 0
$
for degree reasons. 
Applying $\RHom_A(-,\kk(j)[p])$ to the triangle  
$$
A(m)_{\geq -m+1} \xrightarrow{} A(m) \xrightarrow{} \kk(m) \xrightarrow{} A(m)_{\geq -m+1}[1],
$$
we conclude $\Hom_{\D_{\gr}(A)}(\k(m), \k(j)[p]) \cong \Hom_{\D_{\gr}(A)}(A(m), \k(j)[p]) \cong H(\k)_{j - m}^p$ when $j \ge m$. 
\end{proof}

\begin{proof}[Proof of Theorem~\ref{main}]
  We will use the notation of Lemma~\ref{lemma:sod2} throughout the proof. 
 Let us now prove (1) and (3). Since $a \ge 0$, we have $\mathcal{P}_{\geq i+a}\subseteq \: ^{\perp}\mathcal{S}_{<i}$. Thus, the components of the decomposition $\mathcal{D}_i^{\perp} = \langle \mathcal{P}_{\ge i + a}, \mathcal{S}_{< i} \rangle$ from Lemma~\ref{lemma:sod2} may be interchanged, i.e.  
 $\mathcal{D}_i^{\perp} = \langle \mathcal{S}_{< i}, \mathcal{P}_{\ge i + a} \rangle$. In particular, if $a=0$, then the right orthogonals of $\mathcal{D}_i$ and $\mathcal{T}_i$ in $\Dbgr(A)$ coincide and hence we obtain an equivalence $\Dsing(A) \xra{\simeq} \Dqgr(A)$. Noting that $\mathcal{P}_{\geq i}=\langle \mathcal{P}_{\geq i+a},A(-i-a+1),\dots,A(-i)\rangle$ when $a>0$, applying Lemma~\ref{lemma:sod2} once again, we have
 $$
 \Dbgr(A) = \langle \S_{< i}, \mathcal{P}_{\ge i}, \mathcal{T}_i \rangle 
 = \langle \S_{< i}, \mathcal{P}_{\geq i+a},A(-i-a+1),\dots,A(-i), \mathcal{T}_i\rangle.
$$
 We conclude that $\mathcal{D}_i = \langle A(-i-a+1),\dots,A(-i), \mathcal{T}_i\rangle$ when $a>0$. Recall that $\mathcal{D}_i$ is the essential image of $\dGamma_{\ge i}$; hence, by \Cref{prop:faithful}, $\pi_i: \mathcal{D}_i\to \Dqgr(A)$ is an equivalence. Letting $\Phi_i \co \Dsing(A) \to \DQgr(A)$ denote the fully faithful embedding given by the composition $\Dsing(A) \simeq \mathcal{T}_i \into \mathcal{D}_i \simeq \DQgr(A)$, we have the weak semi-orthogonal decomposition
 $
 \DQgr(A) = \langle \pi A(-i-a+1),\dots,\pi A(-i), \Phi_i\Dsing(A)\rangle
 $ when $a>0$. By \Cref{lem:exceptional}, the sequence of objects $\pi A(-i-a+1),\dots,\pi A(-i)$ in $\Dqgr(A)$ form an exceptional collection. Since the left orthogonal of an admissible category is admissible, and a subcategory generated by an exceptional collection is admissible, we see that the above decomposition is in fact semi-orthogonal. We now prove (2). Since $A$ is Gorenstein, we have
$
\RHom_A(\k(s), A(t)) = \k(a + t - s)[n]
$
for some $n \in \Z$. Since $a < 0$, it follows that, if $M\in \mathcal{S}_{< i}$ and $N\in \mathcal{P}_{\geq i}$, then $\Hom_{\Dbgr(A)}(M,N)=H^0(\RHom_A(M,N))_0=0$.
Thus, the components of the decomposition $\mathcal{T}_i^{\perp} = \langle \mathcal{S}_{<i},\mathcal{P}_{\geq i}\rangle$ from \Cref{lemma:sod2} may be interchanged, i.e. $\mathcal{T}_i^{\perp} =  \langle \mathcal{P}_{\geq i}, \mathcal{S}_{<i}\rangle$. Applying Lemma~\ref{lemma:sod2} once again, and using that $\S_{< i - a} = \langle \S_{< i}, \k(-i), \dots, \k(-i + a + 1)\rangle$, we have
$$
\Dbgr(A) = \langle \mathcal{P}_{\ge i}, \S_{< i - a}, \mathcal{D}_{i-a} \rangle
= \langle \mathcal{P}_{\ge i}, \S_{< i}, \k(-i), \dots, \k(-i + a + 1), \mathcal{D}_{i-a} \rangle.
$$
We conclude that $\mathcal{T}_i = \langle \k(-i), \dots, \k(-i + a + 1), \mathcal{D}_{i-a} \rangle$. Recall from \Cref{lemma:sod2} that $q:\mathcal{T}_i\to \Dsing(A)$ is an equivalence. Letting $\Psi_i \co \Dqgr(A) \to \Dsing(A)$ denote the fully faithful embedding given by the composition $\Dqgr(A) \simeq \mathcal{D}_{i - a} \into \mathcal{T}_i \simeq \Dsing(A)$, we have the weak semiorthogonal decomposition
$
\Dsing(A) = \langle q\k(-i), \dots, q\k(-i + a + 1), \Psi_i \Dqgr(A) \rangle.
$
It follows from \Cref{lem:exceptional}(2) that the sequence  $q\k(-i), \dots, q\k(-i + a + 1) \in \Dsing(A)$ forms an exceptional collection. Once again, since the left orthogonal of an admissible category is admissible, and a subcategory generated by an exceptional collection is admissible, the above decomposition is semi-orthogonal.
\end{proof}

\begin{rem}
Let $A$ be as in Theorem~\ref{main} and $i \in \Z$. When $a \le 0$, the embedding $\Psi_i \co \Dqgr(A) \to \Dsg(A)$ is straightforward to describe: it sends  $\w{M} \in \Dqgr(A)$ to $q \dGamma_{\ge i-a}(\w{M}) \in \Dsg(A)$, where $q \co \Db(A) \to \Dsg(A)$ denotes the canonical map. When $a \ge 0$, the embedding $\Phi_i \co \Dsg(A) \to \Dqgr(A)$ is given as follows; this discussion mirrors that of Burke-Stevenson in \cite[Section 5]{BS}. Let $q(M)$ be an object in $\Dsg(A)$. Given a free dg-$A$-module $P$, let $P_{\prec j}$ denote the dg-submodule given by summands of the form $A(s)[t]$ with $s > -j$, and let $P_{\succcurlyeq j} \ce P /P_{\prec j}$.
Let $F$ be a semi-free resolution of $M$ as in \Cref{prop:semi-free} and $G$ a semi-free resolution of $\uHom_A(F_{\succcurlyeq i}, A)$ as in \Cref{prop:semi-free}. The object $\Phi_i(M) \in \Dqgr(A)$ is $\pi(\uHom_A(G,A)_{\prec i})$, where $\pi \co \Db(A) \to \Dqgr(A)$ is the canonical map. For example, if $M \in \Perf(A)$, then $F$ may be chosen to be finitely generated, and $\uHom_A(F_{\succcurlyeq i}, A)$ is its own semi-free resolution. Thus, $\Phi_i(M) = (F_{\succcurlyeq i})_{\prec i} = 0$, as expected.
\end{rem}

\begin{rem}
\label{rem:bidual}
The assumption in Theorem~\ref{main} that $A^0$ is Gorenstein and strictly commutative is needed to apply Proposition~\ref{prop:faithful} in the proof of Lemma~\ref{lemma:sod2}. But, as explained in Remark~\ref{rem:bidual1}, this hypothesis in Proposition~\ref{prop:faithful} can be replaced with the assumption that $A^0$ admits a balanced dualizing complex, in the sense of \cite[Definition 4.1]{YEKUTIELI1_dualizing}; the same is thus true of Theorem~\ref{main}. 
\end{rem}

\begin{proof}[Proof of Corollary~\ref{cor:exc}]
It follows by combining \cite[Corollary 3.12]{RS} and \cite[Corollary 5.6]{YEKUTIELI1_dualizing} that $A^0$ admits a balanced dualizing complex. The result therefore follows from Theorem~\ref{main} and Remark~\ref{rem:bidual}. 
\end{proof}

\begin{example} 
Theorem~\ref{main} applies to each of the families of Gorenstein dg-algebras discussed in Examples~\ref{ex:koszulparam} through \ref{ex:nc}. For instance, suppose $K$ is the Koszul complex on homogeneous forms $f_1, \dots, f_c \in \k[x_0, \dots, x_n]$, where $\deg(x_i) = 1$ for all $i$. Let $\cK$ denote the sheaf of dg-algebras on $\PP^n$ associated to $K$. Recall from Example~\ref{ex:koszulparam} that the Gorenstein parameter of $K$ is $a \ce n+1 - \sum_{i = 1}^c\deg(f_i)$. By \Cref{prop:sheaf}, we have $\Dqgr(K) \simeq \Db(\cK)$, and so Theorem~\ref{main}(3) yields fully faithful embeddings between $\Db(\cK)$ and $\Dsing(K)$. As a specific example, say $n = 3$, $c = 2$, and $f_1$ and $f_2$ are given by the (non-regular) sequence $x_0^2-x_0x_3, x_0x_1-x_0x_2$. Since $a = 0$, we have $\Db(\cK) \simeq \Dsg(K)$.
\end{example}

\begin{example}
Consider the exterior algebra $E = \bigwedge_\k(e_0, \dots, e_n)$, considered as a dg-algebra with trivial differential and bigrading given by $\bideg(e_i) = (d_i, 1)$, where $d_i \ge 1$ for all $i$. The dg-algebra $E$ satisfies the conditions in Theorem~\ref{main}, and we have $\RHom_E(\k, E) = \uHom_E(\k, E) = \k(-d)$, where $d = \sum_{i = 0}^n d_i$. Thus, the Gorenstein parameter of $E$ is $-d < 0$; since $\Dqgr(E) = 0$, we conclude that $\Dsg(E)$ has a full exceptional collection, namely $\Dsg(E) = \langle \k, \dots, \k(-d + 1) \rangle$. It is known that Koszul duality yields an equivalence $\Dsg(E) \simeq \Db(\mathcal{P})$, where $\mathcal{P}$ denotes the weighted projective stack with weights $d_0, \dots, d_n$ \cite[Proposition 6.3]{BE}; moreover, $\Db(\mathcal{P})$ is generated by the exceptional collection $\OO(-d+1), \dots, \OO$ by Theorem~\ref{thm:orlov} and \cite[Corollary 2.18]{Orlov2009}. Thus, the full exceptional collection of $\Dsg(E)$ obtained here is a manifestation of Koszul duality. 
\end{example}

\begin{rem}
\label{rem:strong}
    The exceptional collections arising as the orthogonal of $\Dsg(A)$ in $\Dqgr(A)$ in \Cref{main}(1) and of $\Dqgr(A)$ in $\Dsg(A)$ in \Cref{main}(2) need not be strong.
    \begin{enumerate}
    \item
    Suppose the Gorenstein parameter $a$ is positive, and let $\mathscr{C}$ denote the exceptional collection $\pi A(-i-a+1),\dots,\pi A(-i) \in \Dqgr(A)$. The proof of \Cref{main} implies that $\mathscr{C}$ is strong if and only if the exceptional collection $A(-i-a+1),\dots,A(-i) \in \D_{\on{gr}}(A)$ is strong. It is easily seen that this latter collection is strong if and only if $H^p_j(A)=0$ for all $p<0$ and $j\leq a-1$. In particular, it is strong if $A$ is concentrated in cohomological degree zero, i.e. in the context of Orlov's Theorem (\Cref{thm:orlov}). For an example where the collection $\mathscr{C}$ is not strong, let $A$ be the Koszul complex on $x_0^2, x_0x_1 \in S= \kk[x_0,\dots,x_7]$. In this case, $a = 4$ (\Cref{ex:koszulparam}), and $H^{-1}(A)$ is a non-zero cyclic $S$-module generated in internal degree $3=a-1$; thus, $\mathscr{C}$ is not strong. 
\item
    Now suppose $a < 0$. As in (1), the exceptional collection $\mathscr{C}=q\kk(-i),\dots,q\kk(-i+a+1) \in \Dsing(A)$ is strong if and only if the exceptional collection $\kk(-i),\dots,\kk(-i+a+1)$ in $\D_{\on{gr}}(A)$ is strong. The latter collection is strong if and only if $\underline{\Ext}^p_A(\kk,\kk)_j = 0$ for all $p>0$ and $j\geq a+1$. This need not be the case, even in the context of Orlov's Theorem (\Cref{thm:orlov}): taking $A=\kk[x]/(x^3)$, we have $a=-2$, and $\underline{\Ext}^1_A(\k, \k)_{a + 1} = \k$. 

\item While the exceptional collection obtained in (2) above is not always strong, even in the case of \Cref{thm:orlov}, Orlov constructs a ``dual" exceptional collection that \emph{is} strong in his setting: see the proof of \cite[Corollary 2.9]{Orlov2009}. This exceptional collection has a natural analog in our context; however, it need not always be strong. In detail: we define $E_i \ce A(i+a+1)/A(i+a+1)_{\geq -a}$, and we consider the collection of objects $\mathscr{C} = E_0,\dots,E_{-a-1}$ in $\Dbgr(A)$; one easily checks that this collection generates the same thick subcategory of $\Dbgr(A)$ as $\kk(-i),\dots,\kk(-i+a+1)$. Once again, this is a (strong) exceptional collection if and only if $qE_0,\dots,qE_{-a-1}$ is a (strong) exceptional collection in $\Dsing(A)$. Fix $p \in \Z$ and $0 \le i, j \le -a-1$. We claim that there is an isomorphism $\Hom_{\D_{\on{gr}}(A)}(E_i,E_j[p]) \cong H^p(E_j)_{-i - a - 1}.$
Indeed, let $G$ be a semi-free resolution of $A(i+a+1)_{\geq -a}$ as in \Cref{prop:semi-free}; we have $\Hom_{\D_{\gr}(A)}(A(i+a+1)_{\ge -a}, E_j[p]) = \Hom_A(G, E_j[p]) = 0$ for degree reasons. Applying $\RHom_A( - , E_j[p])$ to the triangle
    $
A(i+a+1)_{\geq -a} \xrightarrow{} A(i+a+1) \xrightarrow{} E_i\xrightarrow{} A(i+a+1)_{\geq -a}[1]
$
gives the desired isomorphism, and it follows immediately from this isomorphism that the collection  $\mathscr{C}$ is exceptional. To see that the collection $\mathscr{C}$ is not always strong, suppose $A$ is the Koszul complex on $x_0^2, x_0x_1, x_2^3 \in \kk[x_0,x_1,x_2]$. By \Cref{ex:koszulparam}, we have $a=-4$, and a direct calculation shows that
$
\Hom_{\D_{\gr}(A)}(E_0, E_3[-1]) \cong H^{-1}(E_3)_3 = H^{-1}(A)_3 \ne 0.
$
    \end{enumerate}
\end{rem}

\begin{rem}
\label{rem:jorgensen}
Let $A$ be as in Setup~\ref{setup}, and suppose $\Dsing(A) = 0$. Since the object $\k$ is perfect and concentrated in cohomological degree 0, a (bigraded version of a) result of J\o rgensen \cite[Theorem~A]{Jorgensen_amplitude} implies that the cohomology of $A$ is concentrated in cohomological degree 0. The case of Theorem~\ref{main} where $\Dsing(A) = 0$ thus yields no new results beyond those implied by Orlov's Theorem (Theorem~\ref{thm:orlov}).
\end{rem}

\begin{rem}
\label{rem:question}
Let us suppose that $A$ is as in Theorem~\ref{thm:orlov}, with Gorenstein parameter $a = 0$, so that we have  $\Dqgr(A) \simeq \Dsing(A)$. It follows from \cite[Theorem 1.1]{krause2} that $\Dqgr(A)$ (resp. $\Dsing(A)$) is equivalent to the subcategory of compact objects in the homotopy category of complexes of injective objects (resp. the homotopy category of acyclic complexes of injective objects) in the abelian category $\qgr(A)$ defined in \cite{AZ} (resp. the abelian category of graded $A$-modules). This raises the question: is the equivalence in Theorem~\ref{thm:orlov}(3) the induced map on compact objects arising from an equivalence of these larger homotopy categories? If so, does this equivalence extend to the differential graded setting?
\end{rem}

\section{Application to the Lattice Conjecture}
\label{sec:lattice}
\def\top{\on{top}}
\def\ch{\on{ch}}
\def\dg{\on{dg}}

Let $\C$ be a $\mathbb{C}$-linear dg-category, $K_*^{\top}(\C)$ its topological $K$-theory groups~\cite{blanc}, and $HP_*(\C)$ its periodic cyclic homology groups; see e.g. \cite[Section 3]{BWchern} for background on periodic cyclic homology of dg-categories. There is a topological Chern character map $\ch^{\top} \co K_*^{\top}(\C) \to HP_*(\C)$~\cite[Section 4]{blanc}. The Lattice Conjecture~\cite[Conjecture 1.7]{blanc} predicts that the complexified topological Chern character is an isomorphism when $\C$ is smooth and proper. We recall that a dg-category is called \emph{smooth} if it is perfect as a $\C$-$\C$-bimodule, and it is \emph{proper} if the total cohomology of each of its morphism complexes is finite dimensional over $\k$; for instance, if $X$ is separated of finite type over $\C$, then the dg-category of perfect complexes on $X$ is smooth and proper if and only if $X$ is smooth and proper over $\mathbb{C}$~\cite[Proposition 3.31]{orlovsmooth}.

\begin{conj}[The Lattice Conjecture]
\label{conj:lattice}
Suppose $\C$ is smooth and proper. The topological Chern character map $\ch^{\top}$ induces an isomorphism $K_*^{\top}(\C) \otimes_\Z \mathbb{C} \xra{\cong} HP_*(\C)$.
\end{conj}

Motivation for the Lattice Conjecture comes from Katzarkov-Kontsevich-Pantev's work on noncommutative Hodge theory~\cite{KKP}. Specifically: when $\C$ is a smooth and proper dg-category, $K_0^{\on{\top}}(\C) \otimes_\Z \mathbb{Q}$ is believed to provide the rational lattice in the (conjectural) noncommutative Hodge structure on $HP_0(\C)$. While the Lattice Conjecture involves smooth and proper dg-categories, it is known to hold in many cases beyond this setting. The following is a list of families of dg-categories for which the Lattice Conjecture is known:
\begin{enumerate}
\item $\on{Perf}(X)$, for $X$ a derived stack of finite type over $\mathbb{C}$ and either Deligne-Mumford with separated diagonal or of the form $[Y/G]$, where $Y$ is a quasi-separated derived algebraic space of finite type over $\mathbb{C}$, and $G$ is an affine algebraic group with diagonalizable identity component~\cite[Theorem A]{khan} (special cases of this result had previously been obtained in \cite{blanc, HLP, kono});
\item a connected, proper dg-algebra  \cite[Theorem 1.1]{kono};
\item a connected dg-algebra $A$ such that $H_0(A)$ is a nilpotent extension of a commutative $\mathbb{C}$-algebra of finite type \cite[Theorem 1.1]{kono};
\item $\Db(X)$, where $X$ is a quasi-separated derived algebraic space of finite type over $\mathbb{C}$~\cite[Theorem B]{khan} (a special case of this result had previously been obtained in~\cite{BW}).
\end{enumerate}

As an application of Corollary~\ref{cor:exc}, we obtain a family of additional cases of the Lattice Conjecture. 
Before we state our result, we recall that, given a dg-algebra $A$, the categories $\Dbgr(A)$, $\Perf(A)$, and $\Dsing(A)$ arise as the homotopy categories of dg-categories. Specifically: finitely generated dg-$A$-modules form a dg-category, and we may take the dg-quotient~\cite{drinfeld} of this dg-category by its subcategory of exact dg-modules to obtain a dg-category $\Dbgr(A)_{\dg}$ whose homotopy category is $\Dbgr(A)$, i.e. a \emph{dg-enhancement} of $\Dbgr(A)$. Define a dg-enhancement $\Perf(A)_{\dg}$ of $\Perf(A)$ in the same way, and then take the dg-quotient of $\Dbgr(A)_{\dg}$ by $\Perf(A)_{\dg}$ to form a dg-enhancement $\Dsg(A)_{\dg}$ of $\Dsing(A)$. Considering $A$ as a differential $\Z$-graded algebra by forgetting the internal grading, we also form the triangulated categories $\Db(A)$, $\on{Perf}(A)$, and $\D^{\on{sg}}(A)$ whose objects are differential $\Z$-graded $A$-modules with no internal grading. We define dg-enhancements $\Db(A)_{\dg}$, $\on{Perf}(A)_{\dg}$, and $\D^{\on{sg}}(A)_{\dg}$ of these triangulated categories just as above.

\begin{thm}
\label{thm:lattice}
Let $A$ be as in Setup~\ref{setup}. Assume that $A$ is Gorenstein and $\dim_k H(A) < \infty$. The Lattice Conjecture holds for the dg-categories $\Perf(A)_{\dg}$, $\on{Perf}(A)_{\dg}$, $\Dbgr(A)_{\dg}$, $\Db(A)_{\dg}$, $\Dsg(A)_{\dg}$, and $\D^{\on{sg}}(A)_{\dg}$.
\end{thm}

\begin{remark}
Theorem~\ref{thm:lattice} was already known for $\on{Perf}(A)_{\dg}$ by a theorem of Konovalov~\cite[Theorem 1.1]{kono}; see (2) in the above list of known cases of the Lattice Conjecture. Theorem~\ref{thm:lattice} gives an alternative proof of a special case of this result of Konovalov. Corollary~\ref{cor:exc} and Lemma~\ref{lem:exceptional} state that
 $\Dsg(A)$ and $\Perf(A)$ have full exceptional collections, and the proof of Theorem~\ref{main} yields that $\Dbgr(A)$ has a semiorthogonal decomposition whose summands have full exceptional collections; 
thus, the topological Chern character for each of these three categories is just a direct sum of copies of $\ch^{\top} \co K^{\top}_*(\mathbb{C}) \to HP_*(\mathbb{C})$. The main cases of interest in Theorem~\ref{thm:lattice} are thus $\Db(A)_{\dg}$ and $\D^{\on{sg}}(A)_{\dg}$. When $A$ is graded commutative, the Lattice Conjecture was previously proven by Khan for $\Db(A)_{\dg}$ (and consequently for $\D^{\on{sg}}(A)_{\dg}$ as well), via different methods \cite[Theorem B]{khan}.
\end{remark}

\begin{proof}[Proof of Theorem~\ref{thm:lattice}]
By Corollary~\ref{cor:exc} and~\Cref{lem:exceptional} , $\Dsg(A)$ and $\Perf(A)$ have full exceptional collections; it is therefore immediate that the Lattice Conjecture holds for $\Dsg(A)_{\dg}$ and $\Perf(A)_{\dg}$. By the ``2 out of 3" property for the Lattice Conjecture~\cite[Theorem 1.1]{kono}, we conclude that the conjecture holds for $\Dbgr(A)_{\dg}$ as well. Let $\Dbgr(A)_{\dg}/(-)^\Z$ denote the dg-orbit category associated to the grading twist functor $(-)$; see e.g. \cite[Page 1]{tabuada} for the definition of the dg-orbit category. Define $\Perf(A)_{\dg}/( - )^\Z$ and $\Dsg(A)_{\dg} / ( - )^\Z$ similarly. The Hom complex between objects $M$ and $N$ in $\Dbgr(A)_{\dg}/(-)^\Z$ is given by $\bigoplus_{i \in \Z} \Hom_{\Dbgr(A)_{\dg}}(M, N(i)) \simeq \Hom_{\Db(A)}(M, N)_{\dg}$, and similarly for $\Perf(A)_{\dg}/( - )^\Z$. It follows that the canonical functors
\begin{equation}
\label{dgfunctors}
\Perf(A)_{\dg}/( - )^\Z \to \on{Perf}(A)_{\dg}, \quad \Dbgr(A)_{\dg} / ( - )^\Z \to \Db(A)_{\dg}
\end{equation}
induce fully faithful embeddings on homotopy categories. Observe that $\on{Perf}(A)$ is generated by $A$, and $\Db(A)$ is generated by $\k$, since $\dim_\k H(A) < \infty$. Thus, the dg-functors~\eqref{dgfunctors} are essentially surjective up to summands on homotopy categories and hence Morita equivalences. 
Applying \cite[Theorem 1.5]{tabuada} and the ``2 out of 3" property for the Lattice Conjecture, the conjecture holds for $\on{Perf}(A)_{\dg}$ and $\Db(A)_{\dg}$. By the ``2 out of 3" property yet again, the conjecture also holds for $\D^{\on{sg}}(A)_{\dg}$.
\end{proof}

\bibliographystyle{amsalpha}
\bibliography{references}
\Addresses
\end{document}